\documentclass[10pt,notitlepage,twoside,a4paper]{amsart}
\usepackage{color}
\usepackage{verbatim}
\usepackage{ifthen}
\usepackage{amssymb}
\usepackage{amsmath, enumerate}
\usepackage{mathrsfs}

\usepackage[T1]{fontenc}
\usepackage[utf8]{inputenc}

\newtheorem{theorem}{Theorem}[section]

\newtheorem{lemma}[theorem]{Lemma}
\newtheorem{proposition}[theorem]{Proposition}

\newtheorem{definition}[theorem]{Definition}

\newtheorem{remark}[theorem]{Remark}

\numberwithin{equation}{section}

\makeatletter
\@namedef{subjclassname@2020}{%
  \textup{2020} Mathematics Subject Classification}
\makeatother

\def\T{\mathcal T}

\def\mc{\mathrm{mod_C\,}}
\def\e{\mathrm{e}}
\def\deg{\mathrm{deg}}
\def\essinf{\mathrm{essinf}\,}

\title[Age-Structured Epidemic Models]{Topological Degree Methods for Age-Structured Epidemic Models}
\author{Luisa Malaguti}
\address[Luisa Malaguti]{Department of Sciences and Methods for Engineering, University of Modena and Reggio Emilia, I-42122 Italy}
\email[L. Malaguti]{luisa.malaguti@unimore.it}
\author{Stefania Perrotta}
\address[Stefania Perrotta]{Department of Physics Informatics and Mathematics, University of Modena and Reggio Emilia, I-41125 Italy}
\email[S. Perrotta]{stefania.perrotta@unimore.it}



\begin{document}

\maketitle
\begin{abstract}
This paper is devoted to the study of an age-structured SIRS epidemic model, in which a population affected by a disease is divided into susceptible, infected, and removed individuals. We assume that the force of infection may be nonlinear and time-dependent. The model, originally introduced and studied by Iannelli and his co-authors \cite{Iannelli}, can be naturally formulated in an abstract setting and has traditionally been analyzed using fixed point techniques, most often the Banach contraction principle.

\noindent Following the approaches of Inaba \cite{Inaba} and Banasiak \cite{Banasiak}, our investigation is based on the semigroup theory, through which we study the existence of mild (integral) solutions. The main novelty of our work lies in the use of the topological degree for condensing maps (see \cite{KOZ}) instead  of classical fixed-point arguments. We prove the existence of a unique, global, nonnegative solution to the model that satisfies the prescribed initial and nonlocal conditions and takes values in the space $L^1$ with respect to the age variable. Moreover, this solution depends continuously on the initial data.
\end{abstract}
%%%%%%%%%%%%%%%%%%%%%%%%%%%%%%%%%%%%%%%%%%%%%%%%%%%%%%%%%%%%%%%%%%%%%%%%%%%%%%%%%%%%%%%%%%%%%%
%            SEZIONE 1              %
%%%%%%%%%%%%%%%%%%%%%%%%%%%%%%%%%%%%%%%%%%%%%%%%%%%%%%%%%%%%%%%%%%%%%%%%%%%%%%%%%%%%%%%%%%%%%%
\section{Introduction}\label{s:introduction}
This paper deals with the spread of a disease in a population divided into susceptible, infected, and removed individuals. As is well known, \emph{susceptible}  refers to individuals who are not sick but can become infected; \emph{ infected } denotes individuals who have the disease and can transmit it; and \emph{removed} refers to individuals who have been infected and are now immune, dead, or isolated. We assume that all vital rates depend on the age of individuals, which varies over a bounded interval $[0, \omega]$. Thus $\omega<+\infty$ is the maximum age an individual of the population can reach. 

\noindent In the first  part of this work, precisely in Section \ref{s:existence},  we study the classical  SIRS  model  (\cite{Iannelli})

\begin{equation}
\label{modello}
\left\{
\begin{split}
 & \!s_t(a\,,t)+s_a(a\,,t)+\mu(a)s(a\,,t)=\!-\Lambda(a\,,i(\cdot\,,t))s(a\,,t)+\delta(a)i(a\,,t)\\
 & \!i_t(a\,,t)+i_a(a\,,t)+\mu(a)i(a\,,t)=\!\Lambda(a\,,i(\cdot\,,t))s(a\,,t)-(\delta(a)+\gamma(a)))i(a\,,t)\\
 & \!r_t(a\,,t)+r_a(a\,,t)+\mu(a)r(a\,,t)=\!\gamma(a)i(a\,,t)
\end{split}
\right.
\end{equation}
 where $s(a\,,t), \, i(a\,,t)$ and $r(a\,,t)$ respectively denote the age-densities of susceptible, infected, and removed individuals at time $t$ with $t\in[0\,,T]$. Thus, the age-density $n(a\,,t)$ of the whole population at time $t$ is 
 \[
  n(a\,,t):=s(a\,,t)+i(a\,,t)+r(a\,,t).
 \]
 
 The demographic evolution is driven by the birth and death rate, $\beta(a)$ and $\mu(a)$, respectively, and  $\Pi(a):=e^{-\int_0^a \mu(\sigma)\, d\sigma}$ is  the probability of survival of an individual to age $a$. Thus, the assumption
 \begin{equation}\label{eq:mu}
 \int_0^\omega \mu(a)\,da=+\infty
 \end{equation}
 ensures that no individual can live beyond $\omega$. 
 
 \par
 Moreover, $\delta(a)$ is the recovery rate and $\gamma(a)$ is the removal rate, that is, the recovery rate with permanent immunity.  As is clear from \eqref{modello}, 
 individuals who have recovered from the infection  are not necessarily immune.
 The force of infection is given by
 \[
   \Lambda(a\,,i(\cdot\,,t))=\int_0^\omega k(a\,,\sigma)i(\sigma\,,t)\,d\sigma.
 \]
 with $k$ bounded and nonnegative.

\par The system \eqref{modello} is equipped with the initial conditions
\begin{equation} \label{e:IC}
s(a\,,0)=s_0, \,\, i(a\,,0)=i_0, \,\, r(a\,,0)=r_0
 \end{equation}
with $s_0, i_0, r_0 \in L^1[0, \omega]$ and the nonlocal boundary conditions 
 \begin{equation}
 \label{BC}
 \begin{split}
     s(0\,,t) & = \int_0^\omega\beta(a) \left(s(a\,,t)+(1-p)i(a\,,t)+(1-q)r(a\,,t)\right) \,da  \\
     i(0\,,t) & = p \int_0^\omega\beta(a) i(a\,,t)\,da   \\
     r(0\,,t) & = q \int_0^\omega\beta(a) r(a\,,t)\,da   
 \end{split}
 \end{equation}
 with $p,q\in[0\,,1]$. 

 \par In the second part, in  Section \ref{s:nonlinear force}, we investigate a more general SIRS model (see \eqref{modello nonlin} ) where  the force of infection is given  by the nonlinear, time-dependent  term 
 \begin{equation}\label{e:ell}
    \ell\left(t,a,\Lambda(a,i(\cdot, t)\right) 
 \end{equation}

 \par
  The regularity assumptions on the functions appearing \eqref{modello} are discussed in Section \ref{s:existence} while the function $\ell$ is introduced in Section \ref{s:nonlinear force}.

Starting with the seminal work by Kermack and McKendrick \cite{KMcK},  the problem of describing and investigating the evolution in time of a disease in a population with a dynamical model  was considered and studied (see the book by Inaba \cite{Inaba} for references). Only  more recently has the age-dependent epidemic model \eqref{modello} been proposed and studied. The book by Iannelli \cite{Iannelli} documents his many contributions as well as those of his coauthors. 

Most of the discussion in  \cite[Chapter VI]{Iannelli} focuses on the case where the disease does not confer immunity, i.e., $\gamma=0$ and hence only the dynamic of susceptible and infected individuals becomes relevant, leading to the so-called SIS model. In \cite{Iannelli} it is also assumed that the reduction rate $R=:\int_0^{\omega}\beta(a) \Pi(a)\, da=1$. This latter condition implies, in particular,  that the whole population $n(a,t)$ has a finite configuration at infinity, $n_{\infty}(a)$ and the model is investigated under the assumption that the population is always in this asymptotic configuration. The force of infection $\Lambda$ in \cite{Iannelli} may also include the term 
\[
k_0(a)i(a,t), \quad \text{with  }k_0 \in L^{\infty}([0, \omega]).
\]
In this way, their  model can be reduced to a single, functional equation for the variable $i(a,t)$, which can then  be analyzed using  the Banach contraction principle. 

The model discussed in \cite{Inaba} treats the case in which the disease confers immunity, i.e., $\delta=0$. It includes an additional term that accounts for vaccination and assumes a population that is already at its stable age distribution. It is important to note that the analysis in \cite{Inaba} is carried out using a semigroup approach, and mild solutions are obtained.

It is worth mentioning several recent contributions by Banasiak and coauthors, primarily devoted to the analysis of the SIS model \cite{Banasiak, BM2015} (see also \cite{BA}, and \cite{BM}). In all these works, a semigroup approach is employed. In particular, the results in \cite{Banasiak} provide a construction of the semigroup, which is then computed in detail for the SIS model in \cite{MM}.

\smallskip

The study of epidemic models has seen significant growth over the last decade. In addition to age structure, spatial spread of individuals has also been incorporated. This has been done, in particular, by Colombo and coauthors \cite{CGMR2020}; see also \cite{CG2025} and \cite{CGMR2023} for a discussion of a more general modeling framework. In \cite{CGMR2020}, the authors include the dynamics of hospitalized individuals; moreover, as in \cite{Inaba}, they assume a disease that confers permanent immunity. The dispersal of individuals in their environment is also considered in \cite{CCZ}, where a reaction–diffusion SIS model with saturation is studied. The authors assume that the transmission and recovery rates depend on both space and time and study classical solutions of the model. In particular, when these rates are temporally periodic, periodic solutions are obtained.

The model \eqref{modello} is often reformulated within its natural abstract framework and subsequently analyzed using fixed-point techniques, in particular the Banach contraction principle. We adopt the same approach in this paper and rewrite problem \eqref{modello}–\eqref{e:IC}–\eqref{BC} in the form of the following initial value problem associated with a semilinear equation.
\begin{equation}
\label{eq:2:CP}
\left\{
\begin{aligned}
x'(t) & =Ax(t) + f(t\,,x(t)), \quad x\in X=:L^1\left([0\,,\omega]\,,\mathbb{R}^3\right), \, t \in[0, T] \\
x(0) & =x_0
\end{aligned}
\right.
\end{equation}
where  the linear part $A:D(A)\subset X\to X$ is not bounded, $f: [0\,,T]\times X \to X$ and $x_0 \in X$. The nonlocal conditions \eqref{BC} are incorporated into the domain $D(A)$. We refer to Sections~\ref{s:existence} and~\ref{s:nonlinear force} for the precise definition of problem \eqref{eq:2:CP}. Following the technique proposed by Banasiak for the SIS model, in Section~\ref{s:existence} we prove that $A$ generates a nonnegative $C_0$-semigroup $\{S(t)\}_{t \ge 0}$ on the space of continuous functions with values in $L^1\left([0\,,\omega]\,,\mathbb{R}^3\right)$ (see \cite{BA} and \cite[Chap.~1]{Pazy} for the related definitions and properties), and we establish the main properties of this semigroup (see Propositions~\ref{p:semigroup1} and \ref{p:semigroup2}). 
Our main result is contained in Theorem~\ref{t:SIRS2}, which deals with problem \eqref{modello nonlin}–\eqref{e:IC}–\eqref{BC}, where the force of infection is given as in \eqref{e:ell}. We assume a.e. nonnegative initial data $s_0, i_0$ and $r_0$ and we prove the existence of a unique, nonnegative solution in the space  $\mathcal{C}\left([0\,,T]\,,L^{1}\left([0\,,\omega] \,,\mathbb{R}^3\right)\right)$, as well as its continuous dependence on the initial data. Finally, additional regularity on $s_0, i_0$, and $r_0$ ensures that the solution is classical. The proof is given in Section \ref{s:nonlinear force}.
The novelty of our investigation lies in the use of the topological degree for condensing (multi)maps, introduced in \cite{KOZ} (see Section~\ref{s:degree} for a brief presentation). To the best of our knowledge, this method has not previously been applied to the study of epidemiological models. It allows us to avoid any restrictions on the reproduction rate $R$ and to assume that infected individuals may become partly susceptible again and partly immune to the disease. In addition, it allows for a nonlinear, time-dependent force of infection. Moreover, our method permits the inclusion of an additional linear term in the force of infection (as in \cite{Iannelli}), a vaccination rate (as in \cite{Inaba}), and the dynamics of hospitalized individuals (as in \cite{CGMR2020}). In the present discussion, we have omitted these additional terms to focus on illustrating the new topological method employed. For this reason, we begin our discussion in Section~\ref{s:existence} with the classical SIRS model \eqref{modello}–\eqref{e:IC}–\eqref{BC}, featuring a linear force of infection, in order to illustrate the application of this degree in the simplest possible setting (see Theorem \ref{t:SIRS}).
Section ~\ref{s:notations} contains some preliminary results.

\smallskip
The topological degree employed in this paper has previously been applied to the study of periodic and nonlocal solutions of transport equations \cite{MP1}, diffusion processes \cite{BC}, and also in the multivalued setting (\cite{BLM}).

%%%%%%%%%%%%%%%%%%%%%%%%%%%%%%%%%%%%%%%%%%%%%%%%%%%%%%%%%%%%%%%%%%%%%%%%%%%%%%%%%%%%%%%%%%%%%
%            SEZIONE 2              %
%%%%%%%%%%%%%%%%%%%%%%%%%%%%%%%%%%%%%%%%%%%%%%%%%%%%%%%%%%%%%%%
%%%%%%%%%%%%%%%%%%%%%%%%%%%%%%%
\section{Notations and preliminary results}\label{s:notations}

%The aim of this paper is to establish the existence of positive, summable solutions to the model. To this end, the abstract formulation of the problem will be set in 
This part contains some preliminary results for the investigation of the semilinear equation contained in \eqref{eq:2:CP}. In the following, we equip the Banach space $X$ (see \eqref{eq:2:CP}) with its usual norm
    \[
\|\psi\|_1=\|\psi_1\|_1+\|\psi_2\|_1+\|\psi_3\|_1=
\int_0^\omega (|\psi_1(a)|+|\psi_2(a)|+|\psi_3(a)|)\,da,
\]
$\psi\in L^1\left([0\,,\omega]\,,\mathbb{R}^3\right)$, and
 partial order
\[
\psi\leq\phi \quad\iff\quad 
\psi_i(a)\leq\phi_i(a)\,\, \text{for a.e. } a\in[0\,,\omega], \text{ for all } i=1,2,3.
\]
We will denote with $X^+$ the closed convex cone of nonnegative functions: 
\[
X^+=\left\{
   \psi=(\psi_1\,,\psi_2\,,\psi_3)\in L^{1}\left([0\,,\omega]\,,\mathbb{R}^3\right):\,
   \psi_i(a)\geq 0 \text{ a.e.\ in } [0\,,\omega],\, i=1,2,3
 \right\}.
\]

\begin{definition}
  A linear operator $L$ on $X$ is called positive if $Lx \in X^+$ for every $x \in X^+$.
Moreover, a semigroup $\{S(t)\}_{t \ge 0}$ is called positive if $S(t)$ is positive for every $t \ge 0$.
\end{definition}

\par\noindent A function $x\in\mathcal{C}([0\,,T]\,,X)$ with $x(0)=x_0$ is said to be a \emph{classical solution} of \eqref{eq:2:CP} (see, e.g., \cite{Pazy}) if it is continuously differentiable with  $x(t) \in D(A)$ and it satisfies  \eqref{eq:2:CP}$_1$ for all $t\in  (0\,,T)$. \par\noindent Instead,  $x\in\mathcal{C}([0\,,T]\,,X)$ is said a \emph{mild solution} of \eqref{eq:2:CP} if
\begin{equation}
\label{e:ms}
x(t) =S(t)x_0 + \int_0^t S(t-s) f(s,x(s)) ds
\end{equation}
for every $t\in[0\,,T]$. \par\noindent  When $f$ is continuous,   every classical solution is also a mild solution of \eqref{eq:2:CP} while the converse is not true, in general. 
\par
The following result deals with the local existence and uniqueness of mild and classical solutions of \eqref{eq:2:CP}. It is a consequence of \cite[Thm 1.4 and Thm 1.5, Chap 6]{Pazy}.
\begin{theorem}
    \label{t:regularity}
 Let $A$ be the infinitesimal generator of a $ C_0$-semigroup $ \{ S(t) \}_{t \geq 0} $. If $f \colon  [0,T]\times X \to X$ is continuously differentiable, then the Cauchy problem \eqref{eq:2:CP} admits locally one and only one mild solution. Moreover, if $x_0\in D(A)$, then the mild solution is a  classical solution.
\end{theorem}

In order to show that the solution of \eqref{modello} takes nonnegative values, we need to replace the function $f$ in \eqref{eq:2:CP} with $\mathcal{F}=f+c\mathbb{I}$ where $\mathbb{I}:X\to X$ is the identity map and $c$ a suitable real value. We are, therefore, led to consider the new problem
  \begin{equation}
\label{eq:2:CP-cI}
\left\{
\begin{aligned}
y'(t) & =\mathcal{A}y(t) + \mathcal{F}(t, y(t)) \\
y(0) & =x_0
\end{aligned}
\right.
\end{equation}
where $\mathcal{A}= A-c\mathbb{I}$. The following result state the equivalence between \eqref{eq:2:CP} and \eqref{eq:2:CP-cI}.
\begin{proposition}
\label{p:equivalenza}
The mild solutions of the Cauchy problem \eqref{eq:2:CP} and of the Cauchy problem \eqref{eq:2:CP-cI}
are the same.
\end{proposition}
\begin{proof}
The $ C_0$-semigroup generated by $\mathcal{A}$  is  
 $\{\e^{-ct} S(t) \}_{t \geq 0}$, so $y\in\mathcal{C}([0\,,T]\,,X)$ is a mild solution of \eqref{eq:2:CP-cI} if
 \begin{equation}
\label{e:ms-cI}
y(t) =\e^{-ct}S(t)x_0 + \int_0^t \e^{-c(t-s)}S(t-s) [f(s\,,y(s))+c y(s)] ds 
\end{equation}
for every $t\in[0\,,T]$.
\par
\smallskip

(1)\,\,  Assume that $x$ satisfy \eqref{e:ms},  by Fubini's theorem we have
\[
\begin{split}
 \int_0^t  c\e^{-c(t-s)} & S(t-s)  x(s)\, ds \\
 &= S(t)x_0-\e^{-ct}S(t)x_0  + \int_0^t\left( 1- \e^{-c(t-\tau)}\right)S(t-\tau) f(\tau\,,x(\tau))\,d\tau.
\end{split} 
\]
Therefore,  since $x$ satisfies \eqref{e:ms}, we have that
\[
\begin{split}
 \e^{-ct}S(t)x_0 & + \int_0^t \e^{-c(t-s)}S(t-s) [{f(s\,,x(s))}+cx(s)]\, ds
 = \e^{-ct}S(t)x_0\\
 & + \int_0^t c\e^{-c(t-s)}S(t-s) x(s)\, ds+ \int_0^t \e^{-c(t-s)}S(t-s) f(s\,,x(s))\, ds \\
 = &\, \e^{-ct}S(t)x_0+S(t)x_0-\e^{-ct}S(t)x_0  + \int_0^t \e^{-c(t-s)}S(t-s) f(s\,,x(s))\, ds \\
 & + \int_0^t\left( 1- \e^{-c(t-\tau)}\right)S(t-\tau) f(\tau\,,x(\tau))\,d\tau \\
 =&\, S(t)x_0+\int_0^t S(t-s) f(s\,,x(s))\, ds=x(t)
\end{split}
\]
concluding that $x$ is a mild solution of \eqref{eq:2:CP-cI}.

\par\smallskip

(2)\,\, Now, let $y$ satisfy \eqref{e:ms-cI}. Notice that for every $s,t\in[0\,,T]$, $s\leq t$,
\begin{equation}
    \label{e:integrali}
  \e^{-c(t-s)}=1-\int_s^t{c}\e^{-c(\tau-s)}\,d\tau
  \qquad\text{and}\qquad
  \int_0^t\e^{-c(t-s)}\,ds = \int_0^t\e^{-cs}\,ds.
\end{equation}
So, by the first identity in \eqref{e:integrali} we get:
\[
\begin{split}
 y(t) = & \e^{-ct}S(t)x_0 + \int_0^t \e^{-c(t-s)}S(t-s) [f(s\,,y(s))+cy(s)] ds \\
 = & S(t)x_0+\int_0^t S(t-s)f(s
 \,,y(s))\,ds
     +c\int_0^tS(t-s)y(s)\,ds \\
   & -\int_0^t  c\e^{-c(t-s)}S(t)x_0\,ds- \int_0^t\int_s^t c\e^{-c(\tau-s)} S(t-s)[f(s\,,y(s))+cy(s)]\,d\tau\,ds.  
\end{split}
\]
Applying Fubini's theorem  we obtain:
\[
\begin{split}
 \int_0^t\int_s^t \e^{-c(\tau-s)} S(t-s) & [f(s\,,y(s))+ cy(s)]\, d\tau\,ds  \\
 & = \int_0^t S(t-s)\int_0^s\e^{-c(s-\tau)}S(s-\tau)[f(\tau\,,y(\tau))+cy(\tau)]\,d\tau\,ds.
\end{split}
\]

\par By the second identity in \eqref{e:integrali} we have that
\[
\int_0^t e^{-c(t-s)}\, dsS(t)x_0=\int_0^t e^{-cs}\, dsS(t)x_0=\int_0^t S(t-s)e^{-cs}S(s)x_0\, ds
\]
so we conclude that
\[
\begin{split}
 y(t) = & S(t)x_0+\int_0^t S(t-s)f(s\,,y(s))\,ds  \\
        & +c\int_0^tS(t-s)y(s)\,ds 
          -c\int_0^t\e^{-c(t- s)}S(t)x_0\,ds  \\
        & - c\int_0^t S(t-s)\int_0^s\e^{-c(s-\tau)}S(s-\tau)[f(\tau\,,y(\tau))+cy(\tau)]\,d\tau\,ds \\
      = &  S(t)x_0+\int_0^t S(t-s)f(s\,,y(s))\,ds + 
          c\int_0^tS(t-s)\zeta(s)\,ds 
\end{split}
\]
where, by \eqref{e:ms-cI}, 
\[
\zeta(s)=y(s)- \e^{-cs}S(s)x_0
-\int_0^s\e^{-c(s-\tau)}S(s-\tau)[f(\tau\,,y(\tau))+cy(\tau)]\,d\tau
=0.
\]
The previous computations show that $y$ satisfies \eqref{e:ms}.
\end{proof}

%%%%%%%%%%%%%%%%%%%%%%%%%%%%%%%%%%%%%%%%%%%%%%%%%%%%%%%%%%%%%%%%%%%%%%%%%%%%%%MISURE%DI%NON%COMPATTEZZA%%%%%%%%%%%%%%%%%%%%%%%%%%%%%%%%%%%%%%%%%%%%%%%%%%%%%%%%%%%%%%%%%%%%%%%%%%%%%%%%

%The fixed point theorem on which our discussion is based (see Section~\ref{s:degree}) uses the topological degree for condensing multimaps. We will therefore begin by presenting the topological tools we will use. We start by introducing some measures of non-compactness (m.n.c. for short) in Banach spaces and their related properties.
Our investigation of problem \eqref{eq:2:CP} also requires the use of suitable measures of non-compactness (m.n.c. for short) and of condensing maps.

Given a non empty subset $C$ of a Banach space $E$, the {\it Hausdorff m.n.c.}\ of $C$ (see e.g.~ \cite[Chap.~2]{KOZ}) is the function $\chi:\mathcal{P}(E)\to [0\,,+\infty]$ defined by
\begin{displaymath}
\chi(C)=\inf\left\{
\epsilon>0:\,\,\exists\, x_1,\dots x_k\in X\textrm{ such that } C\subset\bigcup_{i=1}^kB_\epsilon(x_i)
\right\}.
\end{displaymath}
if $C$ is bounded and $\chi(C)=+\infty$ if $C$ is unbounded.
\par
The next properties of $\chi$ will be used in Sections~\ref{s:existence} and~\ref{s:nonlinear force}; they easily follow from the definition and will be stated without proof.
\begin{proposition}
\label{prop:2:chi} If $\chi:\mathcal{P}(E)\to [0\,,+\infty]$ is the Hausdorff m.n.c.\ defined above then
\begin{itemize}
\item[(i)] $\chi(C)=0$ if and only if $C$ is relatively compact;
\item[(ii)] if $C_1\subset C_2\subset E$ then $\chi(C_1)\leq\chi(C_2)$;
\item[(iii)] for every $C_1,C_2\subset E$, $\chi(C_1\cup C_2)\leq\max\{\chi(C_1)\,,\chi(C_2)\}$;
\item[(iv)] for every $C_1,C_2\subset E$, $\chi(C_1+C_2)\leq \chi(C_1)+\chi(C_2)$;
\item[(v)]  if $Y$ is  a Banach space and $\Phi:E\to Y$ is a Lipschitz function with constant $L$, then for every $C\subseteq E$, $\chi(\Phi(C))\leq L\chi(C)$;
\item[(vi)] for every set $C\subset E$, $\chi(C)=\chi\left(\cup_{\lambda\in[0\,,1]}\lambda C\right)$.
\end{itemize}
\end{proposition}
\par
The following result is proved in \cite[Theorem 4.2.2 and Collorary 4.2.4]{KOZ}.

\begin{theorem}
\label{thm:2:KOZ}
Let $\{S(t)\}_{t \geq 0}$ be a $C_0$-semigroup on a Banach space $E$  and $F$ be the linear operator from $L^1([0\,,T]\,,E)$ to $\mathcal{C}([0\,,T]\,,E)$ defined by
\begin{equation}
\label{eq:2:CO}
F(f)(t)=\int_0^tS(t-s)f(s)\,ds, \quad f\in L^1([0\,,T]\,,E) \textrm{ and } t\in[0\,,T].
\end{equation}
Let $q\in L^1(0\,,T)$ and $\{f_n\}_n\subset L^1([0\,,T]\,,E)$ be  such that
\begin{displaymath}
\chi\left(\{f_n(t)\}_n\right)\leq q(t),\qquad \textrm{ for a.e.~$t\in[0\,,T]$. }
\end{displaymath}
Then
\begin{displaymath}
\chi\left(\{F(f_n)(t)\}_n\right)\leq 2L\int_0^t q(s)\,ds,\qquad \textrm{ for every $t\in[0\,,T]$, }
\end{displaymath}
where $L>0$ is such that $\|S(t)\|\leq L$ for every $t\in[0\,,T]$.
If the Banach space $E$ is separable, then
\begin{displaymath}
\chi\left(\{F(f_n)(t)\}_n\right)\leq L \int_0^t q(s)\,ds,\qquad \textrm{ for every $t\in[0\,,T]$. }
\end{displaymath}
\end{theorem}

In the sequel we will consider the following m.n.c.\ on the subsets of continuous function from $[0\,,T]$ to $X=L^1\left([0\,,\omega]\,,\mathbb{R}^3\right)$ (see \cite[Ex.~2.1.4]{KOZ}). For every bounded set $\Omega\subset\mathcal{C}([0\,,T]\,,X)$.
\begin{equation}\label{nu}
\nu(\Omega)=\max_{\{x_n\}_n\subset\Omega}\left( \sup_{t\in[0\,,T]}\e^{-Nt}\chi\left( \{x_n(t)\}_n\right)\,,\mc(\{x_n\}_n) \right)\in\mathbb{R}_+^2
\end{equation}
where $N\in\mathbb{R}$ is a fixed constant and the maximum is taken with respect to the ordering induced by the cone $\mathbb{R}_+^2$
and $\mc$ is the modulus of equicontinuity defined by
\begin{displaymath}
\mc(\Omega)=\lim_{\delta\to 0}\sup_{x\in\Omega}\max_{|t_1-t_2|<\delta}\|x(t_1)-x(t_2)\|.
\end{displaymath}
The m.n.c.\ $\nu$ is regular, that is $\nu(C)=0$ if and only if $C$ is a relatively compact subset of $\mathcal{C}([0\,,T]\,,X)$
\par

\begin{definition}\label{d:condensante}
Given a Banach space $E$, a set $Y\subset E$, a m.n.c.\ $\beta$, the multivalued mappings $F:Y\multimap E$ or the multivalued family $T:Y\times[0\,,1]\multimap E$ are called {\it condensing} with respect to $\beta$, {\it $\beta$-condensing} for short, if for every $\Omega\subseteq Y$
\begin{displaymath}
 \beta\left(F(\Omega)\right)\geq \beta(\Omega)
 \,\,\Rightarrow\,\,
 \Omega \text{ is relatively compact}
\end{displaymath}
or
\begin{displaymath}
 \beta\left(T(\Omega\,,[0\,,1])\right)\geq \beta(\Omega)
 \,\,\Rightarrow\,\,
 \Omega \text{ is relatively compact.}
\end{displaymath}
\end{definition}
\par
We conclude by recalling some regularity properties of a multimap $F: A \multimap B$, where $A$ and $B$ are topological spaces.
\begin{definition}\label{u.s.c.}
$F$ is upper semicontinuous {\rm(u.s.c.)} at a point $x\in A$ if for every open set $U\subseteq B$ such that $F(x)\subset U$, there exists a neighborhood $V(x)$ of $x$ with the property that $F(V(x)) \subset U$. $F$ is upper semicontinuous if it is u.s.c. at every point $x \in A$.
\end{definition}
\begin{definition}\label{comp.cont.}
$F$ is completely continuous  if it is upper semicontinuous and for every bounded set $\Omega\subset A$, $F(\Omega)$ is relatively compact.
\end{definition}
%%%%%%%%%%%%%%%%%%%%%%%%%%%%%%%%%%%%%%%%%%%%%%%%%%%%%%%%%%%%%%%%%%%%%%%%%%%%%%%%%%%%%%%%%%%%%%%%%%%%%%%%%%%%%%%%%%%%%%%%%%%%%%

\par

\medskip

%%%%%%%%%%%%%%%%%%%%%%%%%%%%%%%%%%%%%%%%%%%%%%%%%%%%PUNTI%FISSI%%%%%%%%%%%%%%%%%%%%%%%%%%%%%%%%%%%%%%%%%%%%%%%%%%%%%%%%%%%%%%%%%%

%%%%%%%%%%%%%%%%%%%%%%%%%%%%%%%%%%%%%%%%%%%%%%%%%%%%%%%%%%%%%%%%%%%%%%%%%%
% SEZIONE 3
%%%%%%%%%%%%%%%%%%%%%%%%%%%%%%%%%%%%%%%%%%%%%%%%%%%%%%%%%%%%%%%%%%%%%%%
\section{The relative topological degree for condensing multimaps}\label{s:degree}
In this section, we introduce the relative topological degree for condensing multivalued operators. 
The presentation is based on \cite[Chap.2 and 3]{KOZ} to which we refer for further details.
\par
This degree provides a tool for detecting the existence of fixed points (see Theorem \ref{t:fixed}), which correspond to solutions of suitable partial differential equations or systems—specifically, the systems \eqref{modello} and \eqref{modello nonlin} equipped with their initial and boundary conditions. As such, it constitutes a fundamental ingredient in our analysis of these systems. For potential generalizations of the present discussion, we present the degree here in its original multivalued formulation, although  in our applications it will be employed for single-valued operators.

\smallskip
In this part the symbol  $E$ always stands for an arbitrary Banach space,   $U\subset E$ is an open set and  $K\subseteq E$ is a convex, closed set with $U\cap K \ne \emptyset$. We also consider $U_K:=U\cap K$ with the topology induced by $K$ and denote with $\overline{U}_K$ and $\partial U_K$ respectively the closure and the boundary of $U_K$  in the relative topology of $K$.

\par
\smallskip
As it is known, to every completely continuous map $f \colon \partial U_K \to K$ (i.e. compact and continuous), such that $x\ne f(x)$ for every $x \in \partial U_K$ can be associated an integer defined as the relative  topological degree of $i-f$:
\[
\text{deg}_K(i-f, \partial U_K).
\]
 Here the symbol $i$ stands for the identity map.  This degree has several properties. In particular the \emph{normalization property}, i.e. when $f(x)= x_0$ for all $x \in \partial U_K$, then
\[
\text{deg}_K(i-f, \partial U_K)=\left\{ \begin{array}{rl} 1, &\text{if } x_0 \in  U_K\\
0, &\text{if } x_0 \not \in  U_K
\end{array}
\right.
\]
 
\noindent Let $F \colon X  \multimap E$ with $X\subseteq E$; we recall that $x_0\in X\subseteq E$ is a {\it fixed point} for the multimap $F \colon X\multimap E $ if $x_0\in F(x_0)$. In this case, we write $x_0\in\mathrm{Fix\,}F$. Notice that in the following, we always consider multimaps $F$ with compact and convex values.
We introduce the relative topological degree for a condensing multimap in two steps. First, by the presence of a \emph{single-valued homotopic approximation} (Theorem~\ref{t:single}) we extend this degree to a completely continuous multimap (see Definition~\ref{comp.cont.}) with compact values. Precisely
\begin{theorem} \label{t:single} \ Let $F \colon \partial U_K \multimap K$ be a completely continuous multimap with compact and convex values and $\mathrm{Fix\,}F\cap \partial U_K=\emptyset$. Then $F$ admits a single-valued homotopic approximation, i.e. there exists $f \colon  \partial U_K \to K$ and a completely continuous family $G \colon [0,1] \times \partial U_K \multimap K$ with compact values and $\mathrm{Fix\,}G(\lambda, \cdot)\cap \partial U_K=\emptyset$ for all $\lambda \in [0,1)$ such that $G(0, \cdot)=f$ and $G(1, \cdot)=F$.
\end{theorem}

\begin{definition}\label{d:degCCM} \ Let $F \colon \partial U_K \multimap K$ be a completely continuous multimap with compact values and  $\mathrm{Fix\,}F=\emptyset$. The relative topological degree
\[
\text{deg}_K(i-F, \partial U_K):=\text{deg}_K(i-f, \partial U_K)
\]
where $f$ is an arbitrary single-valued homotopic approximation of $F$.
\end{definition}

\begin{remark} The relative topological degree  in Definition \ref{d:degCCM} is well defined since it is possible to show that two single-valued homotopic approximations of the same completely continuous multifield $F$ as in Definition \ref{d:degCCM} have the same relative topological degree.
\end{remark}

\par\smallskip
Now we are ready to introduce the relative topological degree of a condensing multimap based on the existence of a \textit{compact homotopy approximation} (see Lemma~\ref{l:homotoly}).  We denote by $\beta$ a m.n.c. in $E$ which is monotone and invariant under union of compact sets. 
\begin{definition} \ Let $F_0, F_1 \colon \partial U_K \multimap K$ be u.s.c. $\beta$-condensing multimaps with compact and convex values.   $F_0$ and $F_1$ are said to be $\beta$-homotopic if there exists an u.s.c. $\beta$-condensing family (see Definition~\ref{d:condensante}) $G \colon [0,1] \times \partial U_K \multimap K$ with compact and convex values such that $\mathrm{Fix\,}G(\lambda, \cdot)=\emptyset$ for all $\lambda \in [0,1], \, G(0, \cdot)=F_0$ and $G(1, \cdot)=F_1$
\end{definition}
\begin{lemma}\label{l:homotoly} \ The class of u.s.c. multimaps $\beta$-homotopic to a given $\beta$-condensing multimap $F\colon \partial U_K \multimap K$ with compact and convex values and $\mathrm{Fix\,}F=\emptyset$ contains a completely continuous representative $\hat F$. The multimap $\hat{F}$ is said a compact homotopy approximation of $F$.
\end{lemma}
The proof of this result involves the notion of a completely fundamentally restrictible multimap, which we omit here for brevity.
\begin{definition} \label{d:degree} \ Let $F \colon \partial U_K \multimap K$ be a  u.s.c. $\beta$-condensing multimap with compact and convex values and $\mathrm{Fix\,}F=\emptyset$. Then
\[
deg_K(i-F, \partial U_K)=deg_K(i-\hat{F}, \partial U_K)
\]
where $\hat{F}$ is an arbitrary compact homotopy approximation of $F$.
\end{definition} 
\begin{remark} \ Again Definition~\ref{d:degree} is well posed since it is possible to show that two homotopy approximation of the same multimap $F$ as in Lemma \ref{l:homotoly} have the same relative degree. 
\end{remark}

\par \noindent
This relative degree is used in the investigation of fixed points, as shown in the following result.
\begin{theorem} \label{t:fixed} \ Let $F \colon \overline{U}_K \multimap K$ be a $\beta$- condensing multimap with compact and convex values. Assume that $\mathrm{Fix\,}F \cap \partial U_K=\emptyset$. If $\text{deg}_K(i-F, \partial U_K )\ne \emptyset$, then $\emptyset \ne \mathrm{Fix\,}F\subset U_K$.
\end{theorem}

\par\smallskip
It is clear from the previous discussion that the search for a fixed point of a u.s.c. $\beta$- condensing multimap $F$ is highly simplified when $F$ admits a constant single-valued  homotopic approximation $f(x)\equiv x_0 \in E$. In fact, in this case, this investigation essentially consists in checking where $x_0$ is located with respect to the domain of $F$. The following result details this situation. We reduce to this case also in the study of problem \eqref{modello} (see Theorem \ref{t:SIRS}) where $E=\mathcal{C}([a\,,b]\,,X)$. 
\begin{theorem}
\label{punto fisso condensante}
Let $\T:\overline{U}_K\times[0\,,1]\multimap K$ be such that:
\begin{itemize}
\item[(1)] $\T(q\,,\lambda)$ is compact and convex, for every $q\in \overline{U}_K$ and $\lambda \in[0\,,1]$;
\item[(2)] $\T$ is u.s.c.;
\item[(3)] $\T$ is $\beta-$condensing, where $\beta$ is a nonsingular, monotonic m.n.c.\ in $E$;
\item[(4)] $\mathrm{Fix\,}\T(\cdot\,,\lambda)\cap\partial U_K=\emptyset$, for every $\lambda \in(0\,,1)$;
\item[(5)] $\T(\cdot\,,0)\equiv\{x_0\}$, $x_0\in U_K$.
\end{itemize}
Then there exists $x\in \overline{U}_K$ such that $x\in\T(x\,,1)$, i.e. $\mathrm{Fix\,}\T(\cdot\,,1) \ne \emptyset$.
\end{theorem}
\begin{proof}
If $\mathrm{Fix\, }\T(\cdot\,,1)\cap\partial U_K\neq\emptyset$, then the proof is complete. Otherwise, by (4) and (5), we obtain that
\begin{equation}
\label{FixT}
\mathrm{Fix\,}\T(\cdot\,,\lambda)\cap\partial U_K=\emptyset,
\quad 
\text{for every $\lambda \in[0\,,1]$.}
\end{equation}
The map $\T(\cdot\,,0)$ is completely continuous with $x_0 \in U_K$ and by the normalization property
\[
\deg_K\left(i-\T(\cdot\,,0)\,,\partial U_K\right)=1.
\] 
Moreover $\T(\cdot\,,0)$ is a homotopic approximation of $\T(\cdot\,,1)$ and so 
 \[
 \deg_K\left(i-\T(\cdot\,,1)\,,\partial U_K\right)=
\deg_K\left(i-\T(\cdot\,,0)\,,\partial U_K\right)=1.
\]
Hence,  by Theorem~\ref{t:fixed}, we conclude
\[
\emptyset\neq\mathrm{Fix\, }\T(\cdot\,,1)\subset U_K.
\]
\end{proof}

%%%%%%%%%%%%%%%%%%%%%%%%%%%%%%%%%%%%%%%%%%%%%%%%%%%%%%%%%%%%%%%%%%%%%%%%%%%%%%%%%%%%%%%%%%%%%%
%            SEZIONE 4             %
%%%%%%%%%%%%%%%%%%%%%%%%%%%%%%%%%%%%%%%%%%%%%%%%%%%%%%%%%%%%%%%%%%%%%%%%%%%%%%%%%%%%%%%%%%%%%%

\section{Existence of a unique nonnegative solution for the SIRS model}
\label{s:existence}
This section is devoted to the study of the classical SIRS model \eqref{modello} equipped with the initial and boundary condition \eqref{e:IC} and \eqref{BC} (see Theorem~\eqref{t:SIRS}). We suppose  (see Section \ref{s:notations}) that
\begin{equation}\label{eq:x0}
x_0:=(s_0, i_0, r_0) \in X^+
\end{equation}

\par
We assume the following standard regularity conditions:
\begin{itemize}
\item[(H1)] \, $\beta, \gamma,\delta \in L^\infty(0\,,\omega)$,
    \[
    \essinf_{[0\,,\omega]}\beta=\beta_0>0,
    \quad
    \essinf_{[0\,,\omega]}\gamma=\gamma_0>0,
    \quad
    \essinf_{[0\,,\omega]}\delta=\delta_0>0;
    \]
\item[(H2)] \,  $\mu \in L^1_{\mathrm loc}(0\,,\omega)$,
    \[
    \essinf_{[0\,,\omega]}\mu=\mu_0>0;
    \]  
 \item[(H3)] \, $k\in L^\infty\left([0\,,\omega]^2\right)$.  
    \[
    \essinf_{[0\,,\omega]\times[0\,,\omega]}k\geq 0.
    \]  
\end{itemize}
\par\smallskip

%In abstract formulation system \eqref{modello} becomes the Cauchy problem
%\begin{equation}
%\tag{$\mathcal{P}$}
 %   \label{astratta1}
%    \left\{
 %   \begin{split}
  %    & x'(t)=Ax(t)+f(x(t)) \\
  %    & x(0)=x_0
   % \end{split}
  %  \right.
%\end{equation}
 %with $t\in[0\,,T]$, $x:[0\,,T]\to L^1\left([0\,,\omega]\,,\mathbb{R}^3\right)$,
 %$x_0=(s_0\,,i_0\,,r_0)$ and we will show in the following how $A$ and $f$ are defined.
 
As stated in Section \ref{s:introduction}, in its abstract formulation problem \eqref{modello}–\eqref{e:IC}–\eqref{BC} becomes the Cauchy problem \eqref{eq:2:CP}. We now show how $A$ and $f$ are defined.
 
 \par
 The linear operator $A$ is obtained by adding three terms $A=A_1+A_2+A_3$, $A_i:D(A_i)\to L^1\left([0\,,\omega]\,,\mathbb{R}^3\right)$, defined as follows:
\begin{itemize}
    \item $A_1$ is the (distributional) derivation  with respect to $a$:
    \[
       A_1\psi=-\psi', \quad D(A_1)=W^{1,1}\left([0\,,\omega]\,,\mathbb{R}^3\right);
    \]
    \item $A_2$ is the contribution of the death rate:
     \[
       A_2\psi=-\mu\psi, \quad D(A_2)=\left\{\psi\in L^{1}\left([0\,,\omega]\,,\mathbb{R}^3\right):\,\mu\psi\in L^{1}\left([0\,,\omega]\,,\mathbb{R}^3\right)\right\};
    \]
    \item $A_3$ is the contribution of the removal and the recovery rates
    \[
       A_3\psi=\mathcal{G}\psi, \quad D(A_3)=L^1\left([0\,,\omega]\,,\mathbb{R}^3\right)
    \]
    where
    \[
    \mathcal{G}=\left( \begin{array}{ccc}
       0 & \delta           & 0 \\
       0 & -(\delta+\gamma) & 0 \\
       0 & \gamma           & 0
     \end{array} \right).
    \]
\end{itemize}
The domain of $A$ is given by the functions $\psi\in D(A_1)\cap D(A_2)\cap D(A_3)$ satisfying the nonlocal boundary conditions \eqref{BC}, that is $\psi(0)=\mathcal{B}\psi$, where $\mathcal{B}:L^{1}\left([0\,,\omega]\,,\mathbb{R}^3\right)\to L^{1}\left([0\,,\omega]\,,\mathbb{R}^3\right)$ is defined by
\[
\mathcal{B}\psi=\int_0^\omega B(a)\psi(a)\,da,\quad 
B=\left( \begin{array}{ccc}
       \beta & (1-p)\beta & (1-q)\beta \\
       0     & p\beta     & 0 \\
       0     & 0          & q\beta
     \end{array} \right).
\]
Therefore
\[
 D(A)=\left\{\psi\in W^{1,1}\left([0\,,\omega]\,,\mathbb{R}^3\right):\,\mu\psi\in L^{1}\left([0\,,\omega]\,,\mathbb{R}^3\right)\,\text{ and }\,\psi(0)=\mathcal{B}(\psi).
 \right\}
\]
\begin{proposition}\label{p:dense} The set $D(A)$ is dense in $L^{1}\left([0\,,\omega]\,,\mathbb{R}^3\right)$.
\end{proposition} 
\begin{proof} In Lemma 3.1.6 in \cite{MM}, the result is proven for a system of two equations (SIS model). The proof in our  case  is entirely analogous.
\end{proof}

\par\smallskip
\begin{proposition}\label{p:semigroup1} The linear operator $A$ generates a strongly continuous semigroup $\{S(t)\}_{t\geq 0}$.
\end{proposition}
\begin{proof} The operator $A_3$ is linear and bounded from $L^{1}\left([0\,,\omega]\,,\mathbb{R}^3\right)$ into itself. Hence, by the theorem of perturbations by bounded linear operators (see e.g. \cite[Theorem 1.1 p. 76]{Pazy}), we only have to show that $A_1+A_2$ generates a strongly continuous semigroup. For a similar model to \eqref{modello}, but with $r(t)=0$ and hence also $q=0$, the proof is contained in \cite{MM}. We follow the same reasoning which involves the use of Hille-Yoshida Theorem (see, e.g. \cite[Corollary 3.6 p. 76]{EngelNagel}). Notice that $D(A_1+A_2)=D(A)$ is dense in $L^{1}\left([0\,,\omega]\,,\mathbb{R}^3\right)$ (see Proposition \ref{p:dense}) and it is easy to see that $A_1$ and $A_2$ are closed in $D(A)$. Assume, in the following, that 
\begin{equation}\label{eq:lambda}
\lambda  > \|\beta\|_{\infty} -\mu_0.
\end{equation}
We prove that the operator $\lambda I-(A_1+A_2)$ has a bounded inverse, defined in $L^{1}\left([0\,,\omega]\,,\mathbb{R}^3\right)$, for every $\lambda$ satisfying \eqref{eq:lambda}. This inverse is the resolvent, denoted with $R(\lambda, A_1+A_2)$ as usual. 

Given $\varphi=(\varphi_1 , \varphi_2, \varphi_3)\in L^{1}\left([0\,,\omega]\,,\mathbb{R}^3\right)$, the value $\psi=(\psi_1, \psi_2, \psi_3)=R(\lambda, A_1+A_2)(\varphi)$ is the solution of the system 
\[
\left\{
\begin{array}{rl}

\psi_1'(a)+(\lambda+\mu(a))\psi_1(a)=&\!\!\!\varphi_1(a)\\
\psi_2'(a)+(\lambda+\mu(a))\psi_2(a)=&\!\!\!\varphi_2(a)\\
\psi_3'(a)+(\lambda+\mu(a))\psi_3(a)=&\!\!\!\varphi_3(a)
\end{array}
\right.
\]
which satisfies the boundary conditions $\psi(0)=\mathcal{B}\psi$. We obtain that
\begin{equation}
\label{e:risolvente}
\psi_i(a)=C_ie^{-\lambda a -\int_0^a \mu(s)\, ds}+\int_0^a \varphi_i(\alpha)e^{-\lambda (a-\alpha) -\int_{\alpha}^a \mu(s)\, ds}\, d\alpha, \, \, a\in [0, \omega]
\end{equation}
with $i=1,2,3$ and for some constants  $C_i$ that will be determined below.  Since $\mu \in L^1_{\mathrm loc}(0, \omega)$ and satisfies \eqref{eq:mu}, the function
\[
F(a):= e^{-\int_0^a \mu(s)\, ds}
\]
is defined for every $a \in [0, \omega]$, with $F(\omega)=0$, and is absolutely continuous in its domain with $F'(a)=-\mu(a)F(a)$ a.e. Consequently, the function $\mu\psi \in L^{1}\left([0\,,\omega]\,,\mathbb{R}^3\right)$. For $\lambda$ satisfying \eqref{eq:lambda} we can also uniquely determine the constants $C_i, \, i=1,2,3$, by means of the boundary conditions $\psi(0)=\mathcal{B}\psi$. Precisely 
\begin{equation}\label{eq:C23}
\begin{array}{l}
\displaystyle  C_2=\frac{p\int_0^\omega \beta(a)\left[\int_0^a\varphi_2(\alpha)e^{-\lambda (a-\alpha) -\int_{\alpha}^a \mu(s)\, ds}\, d\alpha\right]\,da}{1-p\int_0^\omega \beta(a)\e^{-\lambda a -\int_0^a \mu(s)\, ds}\, \, da},\\
\\
\displaystyle C_3=\frac{q\int_0^\omega \beta(a)\left[\int_0^a\varphi_3 (\alpha)e^{-\lambda (a-\alpha) -\int_{\alpha}^a \mu(s)\, ds}\, d\alpha \right]\,da}{1-q\int_0^\omega \beta(a)\e^{-\lambda a -\int_0^a \mu(s)\, ds}\, \, da},\\
\end{array}
\end{equation}
and 
\begin{equation}\label{eq:C1}
\begin{array}{rl}
     C_1=&\frac{\int_0^\omega \beta(a)\left[\int_0^a\varphi_1(\alpha)e^{-\lambda (a-\alpha) -\int_{\alpha}^a \mu(s)\, ds} \, d\alpha\right]\, da }{1-\int_0^\omega \beta(a)\e^{-\lambda a -\int_0^a \mu(s)\, ds}\, \, da} \\
     \\
&-\,\frac{\int_0^{\omega} \beta(a)\left[ (1-p)\psi_2(a)+(1-q)\psi_3(a) \right]\, da}{1-\int_0^\omega \beta(a)\e^{-\lambda a -\int_0^a \mu(s)\, ds}\, \, da}    
\end{array}   
\end{equation}
Therefore, $\psi \in D(A)$; moreover,  with quite standard computations,  we have that
\[
\|\psi_2\|_1 \le \frac{\|\varphi_2\|_1}{\lambda+\mu_0-p\|\beta\|_{\infty}}, \qquad \|\psi_3\|_1 \le \frac{\|\varphi_3\|_1}{\lambda+\mu_0-q\|\beta\|_{\infty}}, 
\]
and
\[
\|\psi_1\|_1 \le \frac{1}{\lambda+\mu_0-\|\beta\|_{\infty}}\left(\|\varphi_1\|_1 +\frac{(1-p)\|\beta\|_{\infty}\|\varphi_2\|_1}{\lambda+\mu_0-p\|\beta\|_{\infty}}+\frac{(1-q)\|\beta\|_{\infty}\|\varphi_3\|_1}{\lambda+\mu_0-q\|\beta\|_{\infty}} \right).
\]
Hence
\[
\begin{array}{rl}
\displaystyle \|\psi\|_1=&\|\psi_1\|_1+\|\psi_2\|_1+\|\psi_3\|_1 \\
\\\displaystyle \le& \frac{\|\varphi_1\|_1}{\lambda+\mu_0-\|\beta\|_{\infty}}+\frac{\|\varphi_2\|_1}{\lambda+\mu_0-p\|\beta\|_{\infty}}\left( \frac{(1-p)\|\beta\|_{\infty}}{\lambda +\mu_0 -\|\beta\|_{\infty}}+1\right)\\
\\\displaystyle +&\frac{\|\varphi_3\|_1}{\lambda+\mu_0-q\|\beta\|_{\infty}}\left( \frac{(1-q)\|\beta\|_{\infty}}{\lambda +\mu_0 -\|\beta\|_{\infty}}+1\right)\\
\\\displaystyle=& \frac{\| \varphi\|_1}{\lambda +\mu_0 -\|\beta\|_{\infty}}.
\end{array}
\]
The resolvent operator $R(\lambda, A_1+A_2) \colon L^{1}\left([0\,,\omega]\,,\mathbb{R}^3\right) \to D(A)$ is then linear and bounded, with 
\begin{equation}
\label{e:norma risolvete}
\|R(\lambda, A_1+A_2)\|\le \frac{1}{\lambda-(\|\beta\|_{\infty}+\mu_0)}.
\end{equation}
Therefore, by the Hille-Yoshida theorem (see \cite[Corollary 3.6 p. 76]{EngelNagel}), $A_1+A_2$ generates a strongly continuous semigroup $\{S^1(t)\}_{t\geq 0}$ and 
\begin{equation}
\label{e:norma semigruppo}
\|S^1(t)\|\le \e^{(\|\beta\|_\infty+\mu_0)t}, \quad t\geq 0.
\end{equation}
\end{proof}

 The following result shows that the semigroup generated by $A$ is invariant with respect to the cone $X^+$ in  $L^{1}\left([0\,,\omega]\,,\mathbb{R}^3\right)$  defined in Section~\ref{s:notations}.
\begin{proposition} \label{p:semigroup2}   The semigroup $(S(t))_{t\geq 0}$ generated by $A$ is invariant with respect to the cone $X^+$ , that is
\[
S(t)X^+ \subseteq X^+, \text{ for every } t \geq 0 .
\]
\end{proposition}
\begin{proof}
 The semigroup generated by $A_3$ is  $\{\e^{A_3 t}\}_{t\geq 0}$. By the special form of $\mathcal{G}$, we have that $e^{A_3t}\psi$ is defined as follows, 
 \begin{equation}
 \label{e:semigruppo G}
e^{A_3t}\psi=
\begin{pmatrix} 
1 &\frac{\delta}{\delta+\gamma}\left(1-e^{-(\delta+\gamma) t}\right) & 0\\
0 & e^{-(\delta +\gamma)t} & 0\\
0 & \frac{\gamma}{\delta+\gamma}\left(1-e^{-(\delta+\gamma) t}\right) & 1
\end{pmatrix}\psi,
 \end{equation}
for $\psi \in L^{1}\left([0\,,\omega]\,,\mathbb{R}^3\right) $ and $t\in [0, T]$, then $\{\e^{A_3t}\}_{t\geq 0}$ is a positive semigroup. 
\par
 As seen in the proof of the previous theorem,
 the resolvent operator $R(\lambda\,,A_1+A_2)$ defined by \eqref{e:risolvente}, with constants $C_1, \, i=1,2,3$ satisfying \eqref{eq:C23}-\eqref{eq:C1}, is positive, hence
 the semigroup $\{S^1(t)\}_{t\geq 0}$ generated by $A_1+A_2$
is positive (see \cite[pp 97-98]{BA}).
\par
So we can conclude that $S(t)=\e^{A_3t}S_1(t)$ is positive.
\end{proof}
\par
%Notice that, by \eqref{e:semigruppo G}, we have that
%\[
%\|e^{A_3t}\|\leq \e^{\left(\|\gamma\|_\infty+\|\delta\|_\infty\right)t}
%, \qquad \forall t\in[0\,,T].
%\]
Moreover, by \cite[Theorem 1.1, page 76]{Pazy}, \eqref{e:norma semigruppo} and the definition of $\mathcal{G}$, we obtain
\begin{equation} \label{e:costante S}
    \|S(t)\|\leq \e^{\left(\|\beta\|_\infty    +2\|\gamma\|_\infty+2\|\delta\|_\infty+\mu_0\right)t}
    \leq L, \qquad \forall t\in[0\,,T],
\end{equation}
where $L=\e^{\left(\|\beta\|_\infty    +2\|\gamma\|_\infty+2\|\delta\|_\infty+\mu_0\right)T}$.
\par\smallskip
\par\medskip
Finally, the function $f:L^{1}\left([0\,,\omega]\,,\mathbb{R}^3\right)\to L^{1}\left([0\,,\omega]\,,\mathbb{R}^3\right)$ is defined by
\[
f(\psi_1\,,\psi_2\,,\psi_3)(a)
 =\left(-\Lambda(a\,,\psi_2)\psi_1(a)\,,
     \Lambda(a\,,\psi_2)\psi_1(a)\,,0
 \right),
 \quad a\in[0\,,\omega],
\]
Where $\Lambda:[0\,,\omega]\times L^{1}([0\,,\omega])\to\mathbb{R}$ is defined by
\[
 \Lambda(a\,,\zeta)=\int_0^\omega k(a\,,\sigma)\zeta(\sigma)\,d\sigma.
\]
Notice that
\begin{equation}
    \label{Lambda}
    |\Lambda(a\,,\zeta)|\leq \|k\|_\infty \|\zeta\|_1,
\end{equation}
for every $\zeta\in  L^{1}([0\,,\omega]$ and a.e.~ $a\in[0\,,\omega]$.
\par

\begin{proposition}
    The map $f$ is continuously differentiable.
\end{proposition}
\begin{proof}
 We can write 
 \[
 f(\psi_1\,,\psi_2\,,\psi_3)=\left(-h(\psi_1\,,\psi_2)\,,h(\psi_1\,,\psi_2)\,,0\right)
 \]
 where $h:L^{1}(0\,,\omega)\times L^{1}(0\,,\omega)\to L^{1}(0\,,\omega)$ is the bilinear map defined by
 
  \[
 h(\psi_1\,,\psi_2)(a)=\psi_1(a)\Lambda(a\,,\psi_2), \quad a\in[0\,,\omega].
 \]
 By \eqref{Lambda}, 
 \[
 \left\|h(\psi_1\,,\psi_2)\right\|_1\leq 
 \|k\|_\infty \|\psi_1\|_1\|\psi_2\|_1.
 \]
 Since $h$ is bilinear and continuous it is Fr\'echet differentiable and its Frech\'et derivative 
 \[
 dh(\psi_1\,,\psi_2)(\delta_1\,,\delta_2)=h(\psi_1\,,\delta_2)+h(\delta_1\,,\psi_2)
 \]
 is continuous in the space of linear operator from $L^{1}(0\,,\omega)\times L^{1}(0\,,\omega)$ in $L^{1}(0\,,\omega)$ (see, for instance,  \cite{AP}).  
Since $h$ is continuously differentiable, also $f$ is continuously differentiable.
\end{proof}
\par

 Adding up all the equations in \eqref{modello} we obtain that the total population $n$ satisfies the following linear Cauchy problem
\begin{equation}
\label{equazione-n}
\begin{cases}
  n_t(a\,,t)+n_a(a\,,t)=-\mu(a)n(a\,,t) &\\
  n(0\,,t)= \int_0^\omega \beta(a)n(a\,,t)\,da &\\
  n(a\,,0)=n_0(a)  &
\end{cases}
\end{equation}
where 
\begin{equation}
\label{equazione-n_0}n_0=s_0+i_0+r_0.
\end{equation}
Therefore, if $x=(s\,,i\,,r)$ is a classical nonnegative solution of the Cauchy problem \eqref{eq:2:CP}, then $n=s+i+r$  is a classical solution to \eqref{equazione-n} and by \cite[Chap.~1, Thm.~4.2 and Thm.~4.3]{Iannelli} there exists a constant $M=M(T)\geq 0$ independent of $n_0$, such that
\begin{equation}
\label{e:n-stima}
    \|n(\cdot\,,t)\|_1\leq M\|n_0\|_1 
\end{equation}
for every $t\in[0\,,T]$.

%The aim of our discussion is to prove that problem \eqref{astratta1} admits a unique solution taking values in $X^+$, provided that we start from an initial condition in $X^+$. 
The topological methods we use do not apply directly to \eqref{eq:2:CP}, so we will introduce two auxiliary problems. The first of these is the following:
 \begin{equation}
 \label{eq:IVP-f-hat}
 \tag{$\hat{\mathcal{P}}$}
    \left\{
    \begin{split}
      & x'(t)=Ax(t)+\hat f(x(t)) \\
      & x(0)=\hat x
    \end{split}
    \right.
\end{equation}
where $\|\hat x\|_1<\|x_0\|_1+1$ and the nonlinear term $\hat f:L^{1}\left([0\,,\omega]\,,\mathbb{R}^3\right)\to L^{1}\left([0\,,\omega]\,,\mathbb{R}^3\right)$ is defined by
\begin{equation}\label{fnuova}
\hat 
f(\psi)(a)= 
\left(-\Xi\left(\Lambda(a\,,\psi_2)\right)\psi_1(a)\,,
     \Xi\left(\Lambda(a\,,\psi_2)\right)\psi_1(a)\,,0
 \right)
\end{equation}
$a\in[0\,,\omega]$,
where $\Xi:\mathbb{R}\to \mathbb{R}$ is a differentiable function such that
\begin{itemize}
 \item $\Xi(z)=z$ if $|z|\leq \|k\|_\infty M(\|x_0\|_1+1)$;
 \item $|\Xi(z)|\leq \|k\|_\infty M(\|x_0\|_1+2)$ for all $z\in\mathbb{R}$;
 \item $\Xi'$ is a continuous function;
 \item $0\leq \Xi'(z)\leq 2$ for all $z\in\mathbb{R}$.
\end{itemize}
The second problem is  \eqref{astratta2}. 

\begin{proposition}
\label{fcontinua}
  The map $\hat f$ defined in \eqref{fnuova} is continuously differentiable. Moreover, for every $\rho>0$ there exists  $C_\rho>0$ such that
  \begin{equation}
\label{lipschitz}
\|\hat f(\psi)-\hat f(\varphi)\|_1\leq 
C_\rho\|\psi-\varphi\|_1,  
\end{equation}
 for every $\psi,\varphi\in L^{1}\left([0\,,\omega]\,,\mathbb{R}^3\right)$, $\|\varphi\|_1\leq \rho$, and it satisfy the following growth condition
\begin{equation}
\label{sublinear}
\|\hat f(\psi)\|_1\leq 
2c\|\psi\|_1,
\end{equation}
for every $\psi\in L^{1}\left([0\,,\omega]\,,\mathbb{R}^3\right)$ with $c=\|k\|_\infty M(\|x_0\|_1+2)$.
\end{proposition}

\begin{proof}
By \eqref{fnuova} it is enough to prove that  $\hat h:L^{1}\left([0\,,\omega]\,,\mathbb{R}^2\right)\to L^{1}(0\,,\omega)$ defined by
\[
\hat h(\psi_1\,,\psi_2)(a)
=\psi_1(a)\,\Xi\left(\Lambda(a\,,\psi_2)\right),
\quad\text{for a.e.~$a\in[0\,,\omega]$,}
\]
is continuously differentiable.
\par
Let us consider the map $g:L^{1}(0\,,\omega)\to L^{\infty}(0\,,\omega)$, $g(\psi)(a)=\Xi(\Lambda(a\,,\psi))$. 
By the continuity of $\Xi'$, the linearity of $\Lambda$ and  \eqref{Lambda}, for every
$\psi,\delta\in L^{1}(0\,,\omega)$ we have
\[
\begin{split}
  g(\psi+\delta)(a)-g(\psi)(a)  
   & =\Xi(\Lambda(a\,,\psi)+\Lambda(a\,,\delta))-\Xi(\Lambda(a\,,\psi))  \\
   & =\Xi'(\Lambda(a\,,\psi))\Lambda(a\,,\delta)
    +\varepsilon(\psi\,,\delta)(a)
\end{split}
\]
for a.e.~$a\in[0\,,\omega]$, where $\|\varepsilon(\psi\,,\delta)\|_\infty=o\left(\|\delta\|_1\right)$
\par
For every $\psi\in L^{1}(0\,,\omega)$, the linear map
$dg(\psi):L^{1}(0\,,\omega)\to L^{\infty}(0\,,\omega)$, defined by 
\[
dg(\psi)(\delta)(a) =\Xi'(\Lambda(a\,,\psi))\Lambda(a\,,\delta)
\quad\text{for a.e.~$a\in[0\,,\omega]$,}
\]
is continuous since $\|dg(\psi)(\delta)\|_\infty \leq 2\|k\|_\infty \|\delta\|_1$. Therefore $g$ is differentiable.
\par
We have to show that $dg$ is a continuous map from $L^{1}(0\,,\omega)$ to the space of the linear operators from $L^{1}(0\,,\omega)$ to $L^{\infty}(0\,,\omega)$. To do this, consider $\psi\in L^{1}(0\,,\omega)$ and a sequence $\{\psi_n\}_n\subset L^{1}(0\,,\omega)$ converging to $\psi$ in $L^1$-norm. 
\par
Fix $\epsilon>0$, we have to prove that
\begin{equation}
\label{e:diff-g}
\|dg(\psi)-dg(\psi_n)\|
=\sup_{\|\delta\|_1\leq 1}\|dg(\psi)(\delta)-dg(\psi_n)(\delta)\|_\infty<\epsilon
\end{equation}
for $n$ large enough. 
\par
Since $\Lambda(\cdot, \psi_{n})$ converges to $\Lambda(\cdot, \psi)$ in $L^{\infty}(0, \omega)$, there exists a closed, bounded interval $I\subset \mathbb{R}$ such that $\Lambda(a, \psi_{n}), \, \Lambda(a, \psi) \in I$ for every $a \in [0, \omega]$. Hence, by the continuity of $\Xi'$, there exists $\nu>0$ such that
for any $z,\bar z\in I$,
\[
|z-\bar z|<\nu \quad\Rightarrow\quad |\Xi'(z)-\Xi'(\bar z)|<\frac{\epsilon}{\|k\|_\infty}.
\]
Moreover, by \eqref{Lambda},
\[
\|\Lambda(\cdot\,,\psi)-\Lambda(\cdot\,,\psi_n)\|_\infty=
\|\Lambda(\cdot\,,\psi-\psi_n)\|_\infty
\leq \|k\|_\infty\|\psi-\psi_n\|_1<\nu
\]
for $n$ large enough, so we can deduce that, for every $\delta\in L^{1}(0\,,\omega)$, $\|\delta\|_1\leq 1$, and for a.e.~$a\in(0\,,\omega)$,
\[
 |dg(\psi)(\delta)(a)-dg(\psi_n)(\delta)(a)|
 =|\Lambda(a\,,\delta)|\,|\Xi'(\Lambda(a\,,\psi))-\Xi'(\Lambda(a\,,\psi_n))|\leq
 \epsilon
\]
for $n$ large enough, proving \eqref{e:diff-g}.
Finally $\hat h(\psi_1\,,\psi_2)=\psi_1g(\psi_2)$ is differentiable and for every $(\psi_1\,,\psi_2),(\delta_1\,,\delta_2)\in L^{1}\left([0\,,\omega]\,,\mathbb{R}^2\right)$ we have
\[
d\hat h(\psi_1\,,\psi_2)(\delta_1\,,\delta_2)=
g(\psi_2)\delta_1+\psi_1dg(\psi_2)\delta_2.
\]
It remains to prove that $d\hat h$ is continuous. If $\{(\psi^n_1\,,\psi^n_2)\}_n$ is a sequence in $L^{1}\left([0\,,\omega]\,,\mathbb{R}^2\right)$ converging to $(\psi_1\,,\psi_2)$ we have
\[
\begin{split}
 &\|d\hat h(\psi_1\,,\psi_2)-d\hat h(\psi^n_1\,,\psi^n_2)\| \\
 & \qquad  =
 \sup_{\|\delta_1\|_1+\|\delta_2\|_1\leq 1}
 \left\|\left(g(\psi_2)-g(\psi^n_2)\right)\delta_1 
 +\left(\psi_1dg(\psi_2)-\psi^n_1dg(\psi^n_2)\right) \delta_2\right\|_1 \\
 & \qquad \leq 
 \|g(\psi_2)-g(\psi^n_2)\|_\infty+\|dg(\psi_2)\|
 \|\psi_1-\psi^n_1\|_1 
 +\|dg(\psi_2)-dg(\psi_2^n)\|\|\psi^n_1\|_1 .
\end{split} 
\]
By the continuity of $g$ and $dg$,
the last term of the previous inequalities tends to $0$ as $n$  tends to infinity, therefore $d\hat h$ is continuous and $\hat f$ is continuously differentiable.

\smallskip
To prove \eqref{lipschitz}, consider $\psi,\varphi\in L^{1}\left([0\,,\omega]\,,\mathbb{R}^3\right)$ with $\|\psi\|_1\leq \rho$. 
We have that
\[
 \left| \Xi(\Lambda(a\,,\psi_2))-\Xi(\Lambda(a\,,\varphi_2))\right|\leq
 2 \left| \Lambda(a\,,\psi_2)-\Lambda(a\,,\varphi_2)\right|\leq
2\|k\|_\infty \|\psi_2-\varphi_2\|_1,
\]
 for a.e.~$a \in [0, \omega]$.
Moreover
\[
\begin{split}
 & \|\hat f(\psi)-\hat f(\varphi)\|_1= 2\int_0^\omega \left| \Xi(\Lambda(a\,,\psi_2))\psi_1(a)-\Xi(\Lambda(a\,,\varphi_2))\varphi_1(a)\right| \,da \\
 &\qquad \leq  
 2\int_0^\omega | \Xi(\Lambda(a\,,\psi_2))||\psi_1(a)-\varphi_1(a)| \,da \\
 &\qquad\quad+ 2\int_0^\omega \left| \Xi(\Lambda(a\,,\psi_2))-\Xi(\Lambda(a\,,\varphi_2))\right| |\varphi_1(a)|\,da \\
 &\qquad \leq  2 \|k\|_\infty M(\|x_0\|_1+2)\|\psi_1-\varphi_1\|_1+4\|\varphi\|_1\|k\|_\infty \|\psi_2-\varphi_1\|_1 \leq C_\rho\|\psi-\varphi\|_1
\end{split}
\]
for $C_\rho =2\left(c+2\rho\|k\|_\infty\right)$, with $c$ defined in \eqref{sublinear}.
\par
As to inequality \eqref{sublinear}, it follows immediately from the definition of $\hat f$. 
\end{proof}

Unfortunately, this function $\hat f$ does not preserve the positivity of the functions. In fact, when $\psi\in X^+$, 
$-\Xi\left(\Lambda(a\,,\psi_2(a))\right)\psi_1(a)\leq 0$ a.e.,
therefore, in general, $\hat f(\psi)\notin X^+$.
To overcome this problem, we will add to the nonlinear term and subtract from the linear term the same linear operator, and we obtain the following problem equivalent to \eqref{eq:IVP-f-hat}: 
\begin{equation}
    \label{astratta2}
    \left\{
    \begin{split}
      & x'(t)=\mathcal{A}x(t)+\mathcal{F}(x(t)) \\
      & x(0)=\hat x
    \end{split}
    \right.
\end{equation}
 $t\in[0\,,T]$, $x:[0\,,T]\to L^1\left([0\,,\omega]\,,\mathbb{R}^3\right)$,
  $x_0$ as in \eqref{eq:x0}
 \[
  \mathcal{A}=A-c\,\mathbb{I}
 \]
and
\[
\mathcal{F}=\hat f+c\,\mathbb{I}
\]
where $\hat f$ is defined in \eqref{fnuova} and $c$ is the constant in \eqref{sublinear}. In \eqref{astratta2}, $\mathcal{A}$ generates a strongly continuous, positive semigroup $(\mathscr{S}(t))_{t\geq 0}$, $\mathscr{S}(t)=\e^{-ct}S(t)$  (see Proposition \ref{p:semigroup1} and  Proposition \ref{p:semigroup2}) and, by \eqref{e:costante S}, 
 \begin{equation}
     \label{m}
     \|\mathscr{S}(t)\|\leq L
     \qquad \forall t\in[0\,,T].
 \end{equation}
Moreover, by the definition of $\Xi$, for every $\psi\in X^+$
\[
\left\{
\begin{split}
  &  \left(c-\Xi\left(\Lambda(a\,,\psi_2(a))\right)\right)
    \psi_1(a)\geq 0 \\
  &  c\psi_2(a)+\Xi\left(\Lambda(a\,,\psi_2(a))\right)
    \psi_1(a) \geq 0  \\
  &  c\psi_3(a)\geq 0
\end{split}
\qquad\quad \text{for a.e.\ $a\in[0\,,\omega]$,}
\right.
\]
so that $\mathcal{F}(\psi)\in X^+$. According to \eqref{lipschitz}, $\mathcal{F}$ is locally lipschitz-continuous and
  \begin{equation}
\label{lipschitz-F}
\|\mathcal{F}(\psi)-\mathcal{F}(\varphi)\|_1\leq 
(C_\rho+c)\|\psi-\varphi\|_1,  
\end{equation}
for every $\psi,\varphi\in L^{1}\left([0\,,\omega]\,,\mathbb{R}^3\right)$, $\|\psi\|_1\leq \rho$, 
and by \eqref{sublinear}  also the sublinearity of the nonlinear term is preserved:
\begin{equation}
    \label{sublinearF}
    \|\mathcal{F}(\psi)\|_1
    \leq 
    3c\,\|\psi\|_1,
\end{equation}
for every $\psi\in L^{1}\left([0\,,\omega]\,,\mathbb{R}^3\right)$. 
\par
We recall that a function $x\in\mathcal{C}\left([0\,,T]\,,L^{1}\left([0\,,\omega] \,,\mathbb{R}^3\right)\right)$ is a mild solution of \eqref{astratta2} if
\begin{equation*}
  x(t) =\mathscr{S}(t)\hat x + \int_0^t \mathscr{S}(t-s) \mathcal{F}(x(s))\, ds ,
  \qquad t\in[0\,,T].
\end{equation*}
By Proposition~\ref{p:equivalenza} the mild solutions of \eqref{astratta2} are the mild solutions of \eqref{eq:IVP-f-hat}.
\par
In the following result, we prove the existence, uniqueness, continuous dependence on initial data and regularity of nonnegative {\it mild solutions} of the models \eqref{eq:IVP-f-hat}.

\begin{theorem}
\label{t:modificato}
For every initial condition $\hat x=(\hat s\,, \hat i\,, \hat r)\in X^+$, $\|\hat x\|_1<\|x_0\|_1+1$, the Cauchy problem \eqref{eq:IVP-f-hat} 
\begin{itemize}
    \item admits a unique mild solution $x\in\mathcal{C}\left([0\,,T]\,,L^{1}\left([0\,,\omega] \,,\mathbb{R}^3\right)\right)$ with nonnegative components;
    \item if $\hat x\in D(A)$, the mild solution is a classical solution;
    \item if $\{\hat x^j\}_j\subset L^{1}\left([0\,,\omega] \,,\mathbb{R}^3\right)$ is a sequence of nonnegative initial data, $\|\hat x_j\|_1< \|x_0\|_1+1$, converging to $\hat x$, therefore the corresponding mild solutions converge uniformly to $x$. 
\end{itemize}
\end{theorem}

\begin{proof}
\textbf{Existence}.
By the previous remark, it is equivalent to show that \eqref{astratta2} admits a mild solution.
 For
\[
R= L\,(\|x_0\|_1+1)\,\e^{3\,c\,L\,T}, 
\]
where $L$ and $c$ are the constants in \eqref{m} and \eqref{sublinear} respectively, let us consider the following subsets of $\mathcal{C}\left([0\,,T]\,,L^{1}\left([0\,,\omega] \,,\mathbb{R}^3\right)\right)$:
\par
the open set defined by
\[
Q=\left\{q \in \mathcal{C}\left([0\,,T]\,,L^{1}\left([0\,,\omega] \,,\mathbb{R}^3\right)\right):\, \sup_{0\leq t\leq T}\|q(t)\|_1<R \right\};
\]
 the closed and convex cone defined by
\[
K=\left\{q \in \mathcal{C}\left([0\,,T]\,,L^{1}\left([0\,,\omega] \,,\mathbb{R}^3\right)\right):\, q(t)\in X^+ \,\,, \forall t\in[0\,,T] \right\};
\]
and the set
\[
{Q_K}={Q\cap K}
\]
with the topology induced by $K$ (see Section~\ref{s:notations} for the definition of $X^+$). As in Theorem~\ref{punto fisso condensante}  we denote with $\overline{Q}_K$ and $\partial Q_K$ the closure and the boundary of $Q_K$ in the relative topology.
\par
The mild solutions of \eqref{astratta2} in $K$ are fixed points of the solution operator $\mathcal{T}(\cdot\,,1)$, where
\[\mathcal{T}: \overline{Q}_K
\times[0\,,1]\to K
\]
is defined by
\[
\mathcal{T}(q\,,\lambda)(t)=\mathscr{S}(t)\hat x + \lambda\int_0^t\mathscr{S}(t-s)\mathcal{F}(q(s))\,ds,
\qquad t\in[0\,,T].
\]
 Notice that the images of $\mathcal{T}$ are in $K$ since $\mathscr{S}(t)$ and $\mathcal{F}$ preserve the positivity and $\hat x \in X^+$.
\par
We will prove that $\mathcal{T}$ satisfies all the assumptions of Theorem~\ref{punto fisso condensante} and therefore that $\mathcal{T}(\cdot\,,1)$ has a fixed point.
\par
\medskip
(1)\,\, $\mathcal{T}$ is single-valued, therefore its images are compact and convex sets.
\par
\medskip
(2)\,\, By \eqref{m} and \eqref{lipschitz-F},  for every $q,q_0\in \mathcal{C}\left([0\,,T]\,,L^{1}\left([0\,,\omega] \,,\overline{Q}_K\right)\right)$ and $\lambda, \lambda_0\in [0\,,1]$ we have
\[
\begin{split}
 \| \mathcal{T} (q\,,\lambda)
 & -\mathcal{T}(q_0\,,\lambda_0)\| \\
 & = \left\|\lambda\int_0^t\mathscr{S}(t-s)\mathcal{F}(q(s))\,ds-\lambda_0\int_0^t\mathscr{S}(t-s)\mathcal{F}(q_0(s))\,ds\right\| \\
& \leq |\lambda-\lambda_0|\int_0^T\left\| \mathscr{S}(t-s)\mathcal{F}(q(s))\,ds\right\|_1 \\
&\quad + |\lambda_0| \int_0^T\left\| \mathscr{S}(t-s)\left(\mathcal{F}(q(s))-\mathcal{F}(q_0(s))\right)\right\|_1\,ds \\
& \leq 3LcRT|\lambda-\lambda_0|+L \int_0^T\left\| \mathcal{F}(q(s))-\mathcal{F}(q_0(s))\right\|_1\,ds  \\
& \leq 3LcRT|\lambda-\lambda_0|
 +L (C_R+c) \int_0^T\left\| q(s)-q_0(s)\right\|_1\,ds  \\
& \leq 3LcRT|\lambda-\lambda_0|+L T(C_R+c)\|q-q_0\| \\
& \leq L T(3cR+C_R+c)\left(|\lambda-\lambda_0|+\|q-q_0\|\right).
\end{split}
\]
Therefore $\mathcal{T}$ is Lipschitz-continuous and assumptions (2) in Theorem~\ref{punto fisso condensante} holds.

\par
\medskip
(3)\,\, Let $\nu$ be the m.n.c.\ on the subsets of $\mathcal{C}\left([0\,,T]\,,L^{1}\left([0\,,\omega] \,,\mathbb{R}^3\right)\right)$ defined by \eqref{nu} with  $N>L(C_R+c)$.
We have to prove that  $\mathcal{T}$ is $\nu$-condensing (Definition~\ref{d:condensante}).
\par
Let $\Omega\subset\overline{Q}_K$  be such that
\begin{equation}
\label{condT}
\nu\left(\mathcal{T}(\Omega\times[0\,,1])\right)\geq\nu(\Omega),
\end{equation}
We aim to prove that $\Omega$ is relatively compact.
By  \eqref{nu} there exists a sequence $\{x_n\}_n\subset \T(\Omega\times[0\,,1])$ such that
\begin{displaymath}
\nu\left(\T(\Omega\times[0\,,1])\right)=\left(\sup_{t\in[0\,,T]}\e^{-Nt}\chi\left( \{x_n(t)\}_n\right)\,,\mc(\{x_n\}_n) \right).
\end{displaymath}
Therefore for every $\{w_n\}_n\subset\Omega$
\begin{equation}\label{massimo}
  \left\{
\begin{aligned}
\phantom{.}& \sup_{t\in[0\,,T]}\e^{-Nt}\chi\left( \{x_n(t)\}_n\right) \geq\sup_{t\in[0\,,T]}\e^{-Nt}\chi\left( \{w_n(t)\}_n\right) \\
\phantom{.}& \mc(\{x_n\}_n)\geq  \mc(\{w_n\}_n)  .
\end{aligned}\right.
\end{equation}
By the definition of $\T$, for every $n\in\mathbb{N}$ there exist $q_n\in \Omega$ and $\lambda_n\in[0\,,1]$ such that
\begin{displaymath}
   x_n(t) = \mathscr{S}(t) \hat x + \lambda_n\int_0^t \mathscr{S}(t-s) \mathcal{F}(q_n(s)) ds, \quad t\in[0\,,T].
\end{displaymath}
Applying  Proposition~\ref{prop:2:chi}(v), \eqref{m} and \eqref{lipschitz-F}, and recalling that $\|q_n(t)\|_1\leq R$ for every $t\in[0\,,T]$, we have that 
\begin{equation}
\label{ChiF}
\begin{split}
   \chi\left(\left\{\mathcal{F}(q_n(s)) \right\}_n\right)
  & \leq (C_R+c)\, \chi\left(\left\{q_n(s))\right\}_n\right)\\
  & =(C_R+c)\,\e^{Ns}\e^{-Ns} \chi\left(\left\{q_n(s))\right\}_n 
      \right)  \\
  & \leq (C_R+c)\,\e^{Ns}\sup_{\tau\in[0\,,T]}\e^{-N\tau} \chi\left(\left\{q_n(\tau)) \right\}_n \right)
\end{split}
\end{equation}
for every $s\in[0\,,t]$. Therefore, by Proposition~\ref{prop:2:chi}, Theorem~\ref{thm:2:KOZ}, \eqref{m} and \eqref{ChiF} we obtain
\begin{displaymath}
 \begin{split}
   \chi\left(\{x_n(t)\}_n\right)
   & \leq  \chi\left(\bigcup_{\lambda\in[0\,,1]}\lambda
     \left\{\int_0^t \mathscr{S}(t-s) \mathcal{F}(q_n(s))\, ds\right\}_n\right)\\
   & = \chi\left(\left\{\int_0^t \mathscr{S}(t-s) \mathcal{F}(q_n(s))\,ds\right\}_n\right)\\
   & \leq L (C_R+c)\int_0^t \e^{Ns}\,ds\sup_{\tau\in[0\,,T]}\e^{-N\tau} \chi\left(\left\{q_n(\tau))\right\}_n\right)\\
   & \leq \frac{L (C_R+c)}{N} \,\e^{Nt}\sup_{\tau\in[0\,,T]}\e^{-N\tau} \chi\left(\left\{q_n(\tau))\right\}_n\right)
 \end{split}
\end{displaymath}
For every $t\in[0\,,T]$. Using \eqref{massimo} we deduce that
\begin{displaymath} 
\begin{split}
 \sup_{t\in[0\,,T]}\e^{-Nt}\chi\left(\{x_n(t)\}_n\right)
 & \leq \frac{L (C_R+c)}{N} \sup_{\tau\in[0\,,T]}\e^{-N\tau} \chi\left(\left\{q_n(\tau))\right\}_n\right)\\
 & \leq \frac{L (C_R+c)}{N} \sup_{t\in[0\,,T]}\e^{-Nt}\chi 
   \left(\{x_n(t)\}_n\right) .
 \end{split}
\end{displaymath}
Since $\frac{L(C_R+c)}{N} <1$, the last inequality holds  if and only if
\begin{equation}
\label{Chinulla}
    \sup_{t\in[0\,,T]}\e^{-Nt}\chi\left(\{x_n(t)\}_n\right) 
    =\sup_{t\in[0\,,T]}\e^{-Nt}\chi\left(\{q_n(t)\}_n\right) =0
\end{equation}
\par
By \eqref{sublinearF}, $\mathcal{F}(q_n(\cdot))$ are uniformly bounded in $[0\,,T]$
\[
\|\mathcal{F}(q_n(t))\|_1\leq3\,c\|q_n(t)\|_1\leq 3\, cR
\]
for every $n\in\mathbb{N}$ and $t\in[0\,,T]$.
Moreover, by \eqref{ChiF},
$\chi\{\mathcal{F}(q_n(t))\}_n=0$ for every $t\in[0\,,T]$, hence, by \cite[Theorem~5.1.1]{KOZ}, $\{x_n\}_n$ is relatively compact in $\mathcal{C}([0\,,T]\,,X)$. We conclude that, by the Ascoli-Arzelà Theorem in abstract spaces, $\{x_n\}_n$ is also  equicontinuous. Therefore
\begin{displaymath}
\nu(\Omega)\leq
\nu\left(\T(\Omega\times[0\,,1])\right)=\left(\sup_{t\in[0\,,T]}\e^{-Nt}\chi\left( \{x_n(t)\}_n\right)\,,\mc(\{x_n\}_n) \right)=0,
\end{displaymath}
proving that $\Omega$ is relatively compact in $\mathcal{C}\left([0\,,T]\,,L^{1}\left([0\,,\omega] \,,\mathbb{R}^3\right)\right)$. Therefore also assumption (3) of Theorem~\ref{punto fisso condensante} is proved.
\par
\medskip
(4)\,\, Let $q_0\in \overline{Q}_K$ and $\lambda_0 \in[0\,,1)$ be such that $q_0\in\T(q_0\,,\lambda_0)$:
\begin{equation}\label{q_0}
 q_0(t)=\mathscr{S}(t)\hat x + \lambda_0\int_0^t \mathscr{S}(t-s) \mathcal{F}(q_0(s))\, ds,
 \quad
 t\in[0\,,T].
\end{equation}
We have to prove that $q_0\notin\partial Q_K$, i.e.\ 
$\|q_0(t)\|_1<R$ for every $t\in [0\,,T]$.
\par
By \eqref{m}, \eqref{sublinearF} and  \eqref{q_0}, we have
\[
 \begin{split}
  \|q_0(t)\|_1
  & \leq \|\mathscr{S}(t) \hat x\|_1 + \lambda_0\int_0^t \|\mathscr{S}(t-s) \mathcal{F}(q_0(s))\|_1\, ds\\
  & \leq L\,\|\hat x \|_1+\int_0^t 3\,c\,L\,\|(q_0(s))\|_1\, ds
 \end{split} 
\]
for every $t\in[0\,,T]$. Therefore, by Gr\"onwall's inequality
\[
  \|q_0(t)\|_1\leq  L\,\|\hat x\|_1\,\e^{\int_0^T 3\,c\,L\, ds}
  < L\,(\|x_0\|_1 +1)\,\e^{3\,c\,L\,T}=R
\]
for every $t\in[0\,,T]$, so also assumption (4) of Theorem~\ref{punto fisso condensante} is satisfied.

\par
\medskip
(5)\,\, Finally, for every $t\in[0\,,T]$

\[
\|\mathcal{T}(t\,,0)\|_1=\|\mathscr{S}(t)\hat x\|_1< L(\|x_0\|_1+1)<R,
\]
therefore  $\mathcal{T}(\cdot\,,0)=\mathscr{S}(t)\hat x\in Q_K$.
\par
All the assumptions of Theorem~\ref{punto fisso condensante} are satisfied, so we conclude that $\mathcal{T}(\cdot\,,1)$ has a fixed point, a nonnegative mild solution $x$ of \eqref{astratta2} satisfying $\| x(t)\|\le R$ for all $t\in [0, T]$ with $R$ defined at the beginning of this proof. Recall that $x$ is also a mild solution to \eqref{eq:IVP-f-hat}.
\par\medskip
\textbf{Uniqueness and regularity}. Since $\hat f$ is continuously differentiable, by Theorem~\ref{t:regularity} the mild solution of \eqref{eq:IVP-f-hat} is unique and it is a classical solution if $\hat x=(\hat s\,,\hat i\,, \hat r)\in D(A)$.
\par\medskip
\textbf{Continuous dependence on the initial data}. Let 
$\{\hat x^j\}_j\subset L^{1}\left([0\,,\omega] \,,\mathbb{R}^3\right)$ be a sequence of initial data converging to  $\hat x$. 
\par
Denote by $x^j\in \mathcal{C}\left([0\,,T]\,,L^{1}\left([0\,,\omega] \,,\mathbb{R}^3\right)\right)$ the unique solution of \eqref{eq:IVP-f-hat} that satisfies the initial condition $x^j(0)=\hat x^j$. We claim that the sequence $\{x^j\}_j$ converges uniformly to $x$, solution of \eqref{eq:IVP-f-hat}.

\par
By \eqref{m} and \eqref{lipschitz-F}, for every $t\in[0\,,T]$, we can estimate
\begin{equation}
\label{dipendenza continua}
\begin{split}
 \|x(t)-x^j(t)\|_1
  & \leq \|\mathscr{S}(t)(\hat x-\hat x^j)\|_1 + \int_0^t \|\mathscr{S}(t-s) \left(\mathcal{F}(x(s))-\mathcal{F}(x^j(s))\right)\|_1\, ds\\
 & \leq L \|\hat x-\hat x^j\|_1+L(C_R+c)\int_0^t\|x(s)-x^j(s)\|_1\,ds
\end{split}
\end{equation}
so that, by Gr\"onwall's inequality,
\[
   \sup_{t\in[0\,,T]}\|x(t)-x^j(t)\|_1\leq  L\,\|\hat x-\hat x^j\|_1\,\e^{LT(C_R+c)}\to 0
\]
as $j\to\infty$, proving the claim.
\end{proof}

Now we are able to prove the existence result for the initial SIRS model \eqref{modello}.

\begin{theorem}
\label{t:SIRS} The problem \eqref{modello}-\eqref{e:IC}-\eqref{BC}, under conditions  {\rm (H1)}, {\rm (H2}), {\rm(H3)} and \eqref{eq:x0} admits a unique mild solution in $\mathcal{C}\left([0\,,T]\,,L^{1}\left([0\,,\omega] \,,\mathbb{R}^3\right)\right)$ with nonnegative components. The solution depends continuously on the initial data.  Moreover, if $x_0\in D(A)$, the mild solution is a classical solution.
\end{theorem}
\begin{proof} Let $x\in \mathcal{C}\left([0\,,T]\,,L^{1}\left([0\,,\omega] \,,\mathbb{R}^3\right)\right)$ be the unique solution of \eqref{eq:IVP-f-hat} with initial data $\hat x=x_0$ (see Theorem ~\ref{t:modificato}).
We claim that 
\begin{equation}\label{eq:claim}
f(x(t))=\hat f(x(t))\qquad\text{for every }t\in[0\,,T]. 
\end{equation} 
 hence $x$ is a solution of \eqref{eq:2:CP}, the equivalent abstract formulation of \eqref{modello}--\eqref{e:IC}-\eqref{BC}.

\par
In fact, if $x_0=(s_0\,, i_0\,, r_0)\in D(A)\cap X^+$, then the nonnegative solution $x=(s\,,i\,,r)$ of \eqref{eq:IVP-f-hat} is classical and $n=s+i+r$ satisfies \eqref{e:n-stima} with $n_0=s_0+i_0+r_0$. Since $s_0\,, i_0\,, r_0$ are non negative
\[
\|n_0\|_1=\|s_0+i_0+r_0\|_1=\|s_0\|_1+\|i_0\|_1+\|r_0\|_1
= \|x_0\|_1
\]
and, by \eqref{Lambda},
\[
    |\Lambda(a\,,i(t))|\leq \|k\|_\infty \|i(t)\|_1\leq \|k\|_\infty \|n(t)\|_1\leq  \|k\|_\infty M \|x_0\|_1
    < \|k\|_\infty M \left(\|x_0\|_1+1\right),
\]
for a.e.~$a\in[0\,,\omega]$ and for every $t\in[0\,,T]$.  

\par
Assume now that $x_0=(s_0\,, i_0\,, r_0)\in X^+$ is an arbitrary initial condition. Since $D(A)\cap X^+$ is dense in $X^+$, there is a sequence $\hat x^j=(\hat s^j\,,\hat i^j\,,\hat r^j)\in D(A)\cap X^+$ converging to $x_0$ in $L^1([0\,,\omega]\,,\mathbb{R}^3)$. Without loss of generality, we can take $\|\hat x^j\|<\|x_0\|+1$. The solutions of the corresponding Cauchy problems \eqref{eq:IVP-f-hat}, $x^j=(s^j,i^j,r^j)$, are classical (see Theorem~\ref{t:modificato}) and
\[
    |\Lambda(a\,,i^j(t))|\leq \|k\|_\infty M \|\hat x^j\|_1
    <\|k\|_\infty M(\|x_0\|_1+1),
\]
for all~$j$, a.e.~$a\in[0\,,\omega]$ and every $t\in[0\,,T]$.
 In Theorem~\ref{t:modificato} we proved that $\{x_j\}_j$ converges to $x$ in $L^{1}\left([0\,,\omega] \,, \mathbb{R}^3\right)$, therefore, passing to the limit, we obtain that every  mild solution of \eqref{eq:IVP-f-hat} satisfies
 \[
    |\Lambda(a\,,i(t))|
    <\|k\|_\infty M(\|x_0\|_1+1),
\]
for  a.e.~$a\in[0\,,\omega]$ and every $t\in[0\,,T]$.
\par 
By the last inequality and the definition of $\hat f$ we have
\[ 
 f(x(t))=\hat f(x(t))\qquad\text{for every }t\in[0\,,T],
\]
so we conclude that  $x$  is also  a solution of \eqref{eq:2:CP}. 
\par
Moreover, by Theorem~\ref{t:regularity} the nonnegative solutions of \eqref{modello} are all and only those of \eqref{eq:IVP-f-hat}.
Therefore the proof follows from Theorem~\ref{t:modificato}.  \end{proof}

%%%%%%%%%%%%%%%%%%%%%%%%%%%%%%%%%%%%%%%%%%%%%%%%%%%%%%%%%%%%%%%%%%%%%%%%%%%%%%%%%%%%%%%%%%%%%%
%            SEZIONE 5             %
%%%%%%%%%%%%%%%%%%%%%%%%%%%%%%%%%%%%%%%%%%%%%%%%%%%%%%%%%%%%%%%%%%%%%%%%%%%%%%%%%%%%%%%%%%%%%%

\section{A nonlinear, time-dependent, force of infection }\label{s:nonlinear force}

\bigskip
\noindent 
In this section, we assume a time-dependent force of infection as in \eqref{e:ell}, which is no longer linear with respect to the number of infected individuals. Hence we consider  a model (see~\eqref{modello nonlin}) that is  more general than~\eqref{modello}. However, in this more general case, the same topological method applied in Section~\ref{s:existence} ensures the unique solvability of~\eqref{modello nonlin} with the given initial and boundary conditions (see Theorem~\ref{t:SIRS2}).
\begin{equation}
\label{modello nonlin}
\left\{
\begin{split}
 & s_t(a\,,t)+s_a(a\,,t)+\mu(a)s(a\,,t)=-\tilde\Lambda(t\,,a\,,i(\cdot\,,t))s(a\,,t)+\delta(a)i(a\,,t)\\
 & i_t(a\,,t)+i_a(a\,,t)+\mu(a)i(a\,,t)=\tilde\Lambda(t\,,a\,,i(\cdot\,,t))s(a\,,t)-(\delta(a)+\gamma(a)))i(a\,,t)\\
 & r_t(a\,,t)+r_a(a\,,t)+\mu(a)r(a\,,t)=\gamma(a)i(a\,,t)
\end{split}
\right.
\end{equation}
where
\begin{equation}\label{e:g1}
\tilde\Lambda(t\,,a\,,i(\cdot\,,t))=
\ell\left(t\,,a\,, \int_0^\omega k(a\,,\sigma)i(\sigma\,,t)\,d\sigma\right)=\ell\left(t,a,\Lambda(a,i(\cdot, t)\right)
\end{equation}
(with $\Lambda$ as in Section~\ref{s:existence}) and the map $\ell \colon [0\,,T]\times [0\,, \omega]\times \mathbb{R} \to \mathbb{R}$ satisfying
\begin{equation}
\tag{$\ell$} 
\begin{split} 
& \ell \in C^1([0\,,T]\times [0\,, \omega]\times \mathbb{R} );\\
& \ell(t,a,0)=0, \quad\text{for all $(t,a)$}; \\
& \ell(t,a,y)\geq 0, \quad\text{for all $(t,a)$ and for all $y>0$.}
\end{split}
\end{equation} 
%\begin{itemize}
% \item[($h_1$)] $h(\cdot\,,\cdot\,,x)$ is measurable %for 
 %     every $x\in\mathbb{R}$;     
 %\item[($h_2$)] $h(t\,,a\,,0)=0$  for a.e. $(t\,,a)\in 
  %    [0\,,T]\times [0\,, \omega]$;
% \item[($h_3$)] for every $r>0$ there exists $H_r>0$ %such 
 %     that  
  %    \[\vert h(t\,,a\,, x_1)- h(t\,,a\,, x_2)\vert \le %H_r\vert x_1-x_2\vert,
%      \]
 %     for a.e. $(t\,,a)\in [0\,,T]\times [0\,, \omega]$ and for every $x_1,x_2\in [-r, r]$.
%\end{itemize}
\begin{remark} Notice that,  by ($\ell$), the map $\ell$ is locally Lipschitz continuous. In particular, for every $r>0$, there exists $H_r>0$ such that  
\begin{equation}
    \label{e:lips-ell}
    \vert \ell(t\,,a\,, y_1)- \ell(t\,,a\,, y_2)\vert \le H_r\vert y_1-y_2\vert
\end{equation}
for every $(t\,,a)\in [0\,,T]\times [0\,, \omega]$ and for every $y_1,y_2\in [-r, r]$. Moreover,
\begin{equation}\label{e:stima-ell}
\vert \ell(t\,,a, y)\vert \le r{H}_r , 
\end{equation}
for every $(t\,,a\,,y)\in [0\,,T]\times [0\,, \omega]\times [-r, r]$.
\end{remark}
We introduce $\tilde h $ defined in $ [0,T]\times L^{1}\left([0\,,\omega] \,,\mathbb{R}^2\right)$ by 
\begin{equation}\label{e:tilde-h}
\tilde h(t\,,\psi_1\,, \psi_2)(a)=\ell\left(t\,,a\,, \Lambda(a, \psi_2)\right)\psi_1(a)
\end{equation}
for a.e.~$a\in [0\,,\omega]$.
\begin{proposition}
\label{p:C}
Let $\tilde h$ be the map defined in \eqref{e:tilde-h}; we have:
\begin{enumerate}
    \item $\tilde h(t\,,\psi_1\,, \psi_2)\in L^1(0\,,\omega)$, for every~$t\in [0\,,T]$ and every $\psi_1, \psi_2 \in L^1(0, \omega)$;
    \item  $\tilde h$ is continuously differentiable.
\end{enumerate}
\end{proposition}
\begin{proof} (1) Let us consider $\psi_1, \psi_2 \in L^1[0, \omega]$.  By Fubini's theorem the map $a \longmapsto \Lambda(a\,,\psi_2)$ is mesurable. Therefore, for every $t\in[0\,,T]$, the map $a \longmapsto \ell(t\,,a,\Lambda(a\,,\psi_2))$ is measurable because composition of a continuous function with a measurable function. By setting 
\[
r:=\Vert k\Vert_{\infty}\Vert \psi_2\Vert_1,
\]
from \eqref{e:stima-ell} we get
\begin{equation}\label{e:stimaC}
\left \vert \tilde h(t\,,\psi_1, \psi_2)(a)\right\vert \le r H_r\vert \psi_1(a)\vert, \quad a \in [0, \omega],
\end{equation}
therefore $\tilde h(t\,,\psi_1, \psi_2)\in L^1(0, \omega)$ for  every $t\in[0\,,T]$. 
\par\smallskip
(2) Following the reasoning used in the proof of Proposition \ref{fcontinua}, the differentiability of 
$\tilde h$ can be shown by proving the continuous  differentiability of the  map $\mathcal{L}:[0\,,T]\times L^1(0\,,\omega)\to L^\infty(0\,,\omega)$ defined by
\[
\mathcal{L}(t\,,\psi)(a)=\ell\left(t\,,a\,, \Lambda(a, \psi)\right).
\]
Since $\ell_t$ and $\ell_y$ are uniformly continuous on compact sets, for every $t,\sigma\in[0\,,T]$ and for every
$\psi, \delta\in  L^1(0\,,\omega)$ we have
\[
\begin{split}
\mathcal{L}(t+\sigma\,,\psi+\delta)(a)-\mathcal{L}(t\,,\psi)(a)=\ell_t\left(t\,,a\,, \Lambda(a, \psi)\right)\sigma+
 \ell_y\left(t\,,a\,, \Lambda(a, \psi)\right)
 &\Lambda(a\,,\delta) \\
 + \varepsilon(t\,,\psi\,,&\sigma\,,\delta)(a)
\end{split}
\]
for a.e.~$a\in[0\,,\omega]$, where 
$\|\varepsilon(t\,,\psi\,,\sigma\,,\delta)\|_\infty=o(|\sigma|+\|\delta\|_1)$.
\par
The linear map
$d\mathcal{L}(t\,,\psi):\mathbb{R}\times L^1(0\,,\omega)\to L^\infty(0\,,\omega)$ defined by
\[
d\mathcal{L}(t\,,\psi)(\sigma\,,\delta)(a)
=\ell_t\left(t\,,a\,, \Lambda(a, \psi)\right)\sigma+
 \ell_y\left(t\,,a\,, \Lambda(a, \psi)\right)
 \Lambda(a\,,\delta) 
\]
is continuous; in fact
\[
\|d\mathcal{L}(t\,,\psi)(\sigma\,,\delta)\|_\infty
\leq L_1|\sigma|+L_2\|k\|_\infty\|\delta\|_1
\]
where, by \eqref{Lambda}, $L_1=\max\left\{\left|\ell_t(t\,,a\,,y):\,\,t\in[0\,,T],\,a\in[0\,,\omega],\, |y|\leq \|k\|_\infty\|\psi\|_1\right|\right\}$ and $L_2=\max\left\{\left|\ell_y(t\,,a\,,y):\,\,t\in[0\,,T],\,a\in[0\,,\omega],\, |y|\leq \|k\|_\infty\|\psi\|_1\right|\right\}$.
Therefore, $d\mathcal{L}$ is the differential of $\mathcal{L}$ at $(t\,,\psi)$. It remains to prove that $d\mathcal{L}$ depends continuously on $(t\,,\psi)$.
If $\left\{t_n\right\}_n$ is a sequence in $[0\,,T]$ converging to $t$ and $\left\{\psi_n\right\}_n$ is a sequence in $L^1(0\,,\omega)$ converging to $\psi$, then
\[
\sup_{a\in[0\,,\omega]}\left|\left(t_n\,,a\,, \Lambda(a, \psi_n)\right)-\left(t\,,a\,, \Lambda(a, \psi)\right)\right|
\leq |t_n-t|+k_\infty\|\psi_n-\psi\|_1\to 0
\]
as $n\to \infty$.  Therefore, from assumption ($\ell$), it follows that $d\mathcal{L}(t_n\,,\psi_n)\to d\mathcal{L}(t\,,\psi)$ as $n\to \infty$ in the norm of linear operators from $\mathbb{R}\times L^1(0\,,\omega)$ to $L^\infty(0\,,\omega)$.
\par
Finally, we have that $\hat h$ is differentiable and 
\[
d \tilde h(t\,,\psi_1\,,\psi_2)(\sigma\,,\delta_1\,,\delta_2)
=\mathcal{L}(t\,,\psi_2)\delta_1
+\psi_1 d\mathcal{L}(t\,,\psi_2)(\sigma\,,\delta_2).
\]
The  continuity of $d \tilde h$ with respect to $(t\,,\psi_1\,,\psi_2)$ can be proved analogously to the continuity of $d\hat h$ in Proposition \ref{fcontinua}.
\end{proof}

In its abstract formulation system~\eqref{modello nonlin} with usual initial conditions \eqref{e:IC} and boundary conditions~\eqref{BC} becomes the Cauchy problem
\begin{equation}
\tag{$\tilde{\mathcal{P}}$}
    \label{astratta1}
    \left\{
    \begin{split}
      & x'(t)=Ax(t)+\tilde  f(t, (x(t)) \\
      & x(0)=x_0
    \end{split}
    \right.
\end{equation}
 with $t\in[0\,,T]$, $x:[0\,,T]\to L^1\left([0\,,\omega]\,,\mathbb{R}^3\right)$,
 $x_0$ and $A$ as in Section~\ref{s:existence} and $\tilde f \colon [0,T]\times L^{1}\left([0\,,\omega]\,,\mathbb{R}^3\right)\to L^{1}\left([0\,,\omega]\,,\mathbb{R}^3\right)$  defined by
\[
\tilde f(t,\psi_1\,,\psi_2\,,\psi_3)(a)
 =\left(-\tilde \Lambda(t,a\,,\psi_2)\psi_1(a)\,,
     \tilde \Lambda(t,a\,,\psi_2)\psi_1(a)\,,0
 \right),
 \quad a\in[0\,,\omega],
\]
\begin{proposition}\label{p:tilde f}
 The map $\tilde f$ is continuous differentiable.  
\end{proposition}
\begin{proof} Since, for $t \in [0,T], \, (\psi_1, \psi_2, \psi_3) \in L^{1}\left([0\,,\omega]\,,\mathbb{R}^3\right)$, 
\[
\tilde f(t, \psi_1, \psi_2, \psi_3)=\left(-\tilde h(t, \psi_1, \psi_2), \tilde h(t, \psi_1, \psi_2), 0\right),
\]
the result is an immediate consequence of Proposition~\ref{p:C}.
\end{proof}

We are now ready to state the existence result for problem \eqref{modello nonlin}.

\begin{theorem}
\label{t:SIRS2} The problem \eqref{modello nonlin}-\eqref{e:IC}-\eqref{BC}, under conditions  {\rm (H1)}, {\rm (H2}), {\rm(H3)}, ($\ell$) and \eqref{eq:x0} admits a unique mild solution in $\mathcal{C}\left([0\,,T]\,,L^{1}\left([0\,,\omega] \,,\mathbb{R}^3\right)\right)$ with nonnegative components. The solution depends continuously on the initial data.  Moreover, if $x_0\in D(A)$, the mild solution is a classical solution.

%Consider the system \eqref{modello nonlin} with boundary conditions \eqref{BC}.
%If the coeﬃcients of the problem  satisfy conditions {\rm (H1)},{\rm (H2}) and {\rm(H3)} and the function $\tilde \Lambda$ (see \eqref{e:g1}) condition ($\ell$),  then the Cauchy problem \eqref{modello nonlin}--\eqref{e:IC} with  initial conditions $x_0=(s_0\,, i_0\,, r_0)\in X^+$ admits a unique mild solution in  $\mathcal{C}\left([0\,,T]\,,L^{1}\left([0\,,\omega] \,,\mathbb{R}^3\right)\right)$ with nonnegative components.
   % \par
%    Moreover, if $x_0\in D(A)$, the mild solution is a classical solution.
\end{theorem}
 \begin{proof} The proof follows the same approach developed in Section \ref{s:existence}. More precisely, we 
show the solvability of problem $(\tilde{\mathcal{P}})$,  which represents the abstract formulation of \eqref{modello nonlin}–\eqref{e:IC} with the boundary condition \eqref{BC}.
 
\noindent We introduce the new function $\hat g \colon [0,T]\times L^{1}\left([0\,,\omega]\,,\mathbb{R}^3\right)\to L^{1}\left([0\,,\omega]\,,\mathbb{R}^3\right)$  defined by
\[
\hat g(t,\psi_1\,,\psi_2\,,\psi_3)(a)
 =\left(-\ell(t,a, \Xi(\Lambda(a, \psi_2))\psi_1(a), \, \ell(t,a, \Xi(\Lambda(a, \psi_2))\psi_1(a), \, 0 \right),
\]
 $a\in[0\,,\omega]$, where $\Xi$ was defined in the previous section. Let 
\[
\hat c=\max\left\{\ell(t\,,a\,,y):\,\, t\in[0\,,T],\,
a\in[0\,,\omega],\, |y|\leq \|k\|_\infty M(\|x_0\|_1+2)\right\}.
\]
Using arguments analogous to those used in the proof of formula \eqref{lipschitz} in Proposition~\ref{fcontinua}, one can show that for every $\rho>0$
\begin{equation}\label{lips hat g}
\|\hat g(t\,,\psi)-\hat g(t\,,\varphi)\|_1
\leq \hat C_\rho\|\psi-\varphi\|_1
\end{equation}
for every $t\in[0\,,T]$ and for every $\psi,\varphi\in L^{1}\left([0\,,\omega]\,,\mathbb{R}^3\right)$, $\|\varphi\|_1\leq \rho$, where 
$\hat C_\rho=\left(2\hat c + 4H_r\|k\|_\infty\rho\right)$ and $r=\|k\|_\infty M(\|x_0\|_1+2)$.
Moreover, the definition of $\hat g$ implies that
\begin{equation*}
\| \hat g(t, \psi)\|_1 \le 2\hat c \| \psi\|_1
\end{equation*}
for $t \in [0,T]$ and $\psi \in L^{1}\left([0\,,\omega]\,,\mathbb{R}^3\right)$.

 \par
 
 We consider the new auxiliary  problems, in abstract space,
 \begin{equation}
 \label{eq:IVP-g-hat}
 \tag{$\hat{\mathcal{Q}}$}
    \left\{
    \begin{split}
      & x'(t)=Ax(t)+\hat g(t\,, x(t)) \\
      & x(0)=\hat x
    \end{split}
    \right.
\end{equation}
and 
\begin{equation}
    \label{astratta3}
    \left\{
    \begin{split}
      & x'(t)=\mathcal{A}x(t)+\mathcal{G}(t\,, x(t)) \\
      & x(0)=\hat x
    \end{split}
    \right.
\end{equation}
with $\mathcal{A}$ defined as in Section \ref{s:existence} and $\hat x \in X^+$ with $\| \hat x\| \le \|x_0\| +1$ (as in the statement of Theorem \ref{t:modificato}), while $\mathcal{G}=\hat g +\hat c\mathbb{I}$. 

\par
The two problems $\hat{\mathcal{Q}}$ and \eqref{astratta3} are equivalent by Proposition~\ref{p:equivalenza}. By the definition of $\hat c$, the function $\mathcal{G}$ preserves the positivity of the functions, i.e. $\mathcal{G}(t,\psi) \in X^+$ for all $t\in [0,T]$ and $\psi \in  X^+$.

The solvability of \eqref{astratta3} follows by analogy with the proof of Theorem ~\ref{t:modificato}.
In particular, with regard to the existence part, we now have to consider the solution operator
\[
\mathcal{T}(q\,,\lambda)(t)=\mathscr{S}(t)\hat x + \lambda\int_0^t\mathscr{S}(t-s)\mathcal{G}(s, q(s))\,ds,
\qquad t\in[0\,,T],
\] 
where now  $q$ belongs to the ball, centered in $0$ of radius $\hat R=L(\|x_0\|_1+1)e^{3\hat cT}$.
Though now $\mathcal{G}$ explicitly depends on $t$, the same reasoning applies, as in the proof of Theorem ~\ref{t:modificato}.

Furthermore, when $x$ and $y$ are two solutions of the equation in  \eqref{astratta3} with initial conditions $x_0$ and $y_0$, respectively, we obtain that
\[
\Vert x(t)-y(t)\Vert_1\leq  \|\mathscr{S}(t)(x_0-y_0)\|_1 + \int_0^t \|\mathscr{S}(t-s) \left(\mathcal{G}(s, x(s))-\mathcal{G}(s,y(s))\right)\|_1\, ds
\]
Therefore, since \eqref{lips hat g} implies the same Lipschitz continuity for $\mathcal G$ and with a reasoning similar to that in \eqref{dipendenza continua}, we obtain the continuous dependence of the solutions of \eqref{astratta3} on the initial data.

\noindent Now let $x\in \mathcal{C}\left([0\,,T]\,,L^{1}\left([0\,,\omega] \,,\mathbb{R}^3\right)\right)$ be the unique solution of ($\hat Q$).We claim that 
\begin{equation}\label{eq:claim}
\tilde f(t\,,x(t))=\hat g(t\,, x(t))\qquad\text{for every }t\in[0\,,T]. 
\end{equation} 
and hence $x$ is a solution of ($\hat P$), the equivalent abstract formulation of  \eqref{modello nonlin}-\eqref{e:IC}.
Again this claim can be proved in a similar way than in Theorem \ref{t:SIRS} so the proof is complete.
 \end{proof}

%%%%%%%%%%%%%%%%%%%%%%%%%%%%%%%%%%%%%%Acknowledgments%%%%%%%%%%%%%%%%%%%%%%%%%%%%%%%%%%%%%%%%%%%%%%%%%%%%%%%%%
\medskip
\textbf{Acknowledgments.} The authors are members of the Gruppo Nazionale per l'Analisi  Matematica, la Probabilit\`{a} e le loro Applicazioni (GNAMPA) of the Istituto Nazionale di Alta Matematica (INdAM) and acknowledge financial support from this institution.
L. Malaguti was partially supported by PRIN 2022 ``Modeling, Control and Games through Partial Differential Equations" (coordinator R.M. Colombo).
%%%%%%%%%%%%%%%%%%%%%%%%%%%%%%%%%%%%%%%%%%%%%%%%%%%%%%%%%%%%%%%%%%%%%%%%%%%%%%%%%%%%
                       %REFERENCES%
%%%%%%%%%%%%%%%%%%%%%%%%%%%%%%%%%%%%%%%%%%%%%%%%%%%%%%%%%%%%%%%%%%%%%%%%%%%%%%%%%%%%

\bibliographystyle{amsplain}
\bibliography{MP-ref}

\end{document}